\documentclass[12pt]{amsart}

\usepackage{framed}
\usepackage{braket}
\usepackage{amsmath,amssymb,amsthm}
\usepackage{color}
\usepackage{bm}
\usepackage[all]{xy}
\usepackage{amscd}
\usepackage{graphicx}
\usepackage{ascmac}
\usepackage{BOONDOX-frak, mathrsfs}
\usepackage[top=30truemm,bottom=30truemm,left=25truemm,right=25truemm]{geometry}
\usepackage{mathtools}
\mathtoolsset{showonlyrefs=true}

\makeatletter
\@addtoreset{equation}{section}

\makeatother

\newcommand{\Z}{\mathbb{Z}}
\newcommand{\Q}{\mathbb{Q}}

\newcommand{\F}{\mathbb{F}}

\newcommand{\bH}{\mathbb{bH}}
\newcommand{\bN}{\mathbb{N}}

\newcommand{\pe}{\mathfrak{p}}

\newcommand{\qu}{\mathfrak{q}}

\newcommand{\vsim}{{\rotatebox{90}{$\sim$}}}
\newcommand{\ttilde}{\widetilde}

\newcommand{\RR}{\mathcal{R}}

\newcommand{\LL}{\mathcal{L}}

\newcommand{\fD}{\mathfrak{D}}

\newcommand{\otimesL}{\otimes^{\mathbb{L}}}

\newcommand{\derR}{\mathsf{R}}
\newcommand{\RG}{\derR\Gamma}
\newcommand{\RHom}{\derR\Hom}

\newcommand{\ol}[1]{\overline{#1}}
\newcommand{\ul}[1]{\underline{#1}}

\newcommand{\iDet}{\mathrm{d}}

\newcommand{\lra}[1]{\langle #1 \rangle}
\newcommand{\parenth}[1]{\left( #1 \right)}

\newcommand{\dSel}[2]{\mathfrak{X}_{#2}^{#1}}

\DeclareMathOperator{\Gal}{Gal}

\DeclareMathOperator{\Ker}{Ker}
\DeclareMathOperator{\Coker}{Coker}

\DeclareMathOperator{\pd}{pd}
\DeclareMathOperator{\Fitt}{Fitt}
\DeclareMathOperator{\Det}{Det}
\DeclareMathOperator{\rank}{rank}
\DeclareMathOperator{\Hom}{Hom}

\DeclareMathOperator{\nt}{nt}
\DeclareMathOperator{\ord}{ord}

\DeclareMathOperator{\Sel}{Sel}

\DeclareMathOperator{\ur}{ur}
\DeclareMathOperator{\tor}{tor}
\DeclareMathOperator{\ram}{ram}

\DeclareMathOperator{\an}{an}
\DeclareMathOperator{\Iw}{Iw}

\DeclareMathOperator{\Ext}{Ext}

\DeclareMathOperator{\length}{length}
\DeclareMathOperator{\PN}{PN}

\DeclareMathOperator{\Tr}{Tr}
\DeclareMathOperator{\Loc}{Loc}
\DeclareMathOperator{\rel}{rel}

\DeclareMathOperator{\ssr}{ss}

\let\oldenumerate\enumerate
\renewcommand{\enumerate}{
   \oldenumerate
   \setlength{\itemsep}{1pt}
   \setlength{\parskip}{0pt}
   \setlength{\parsep}{0pt}
}
\let\olditemize\itemize
\renewcommand{\itemize}{
   \olditemize
   \setlength{\itemsep}{1pt}
   \setlength{\parskip}{0pt}
   \setlength{\parsep}{0pt}
}

\theoremstyle{plain}
\newtheorem{thm}{Theorem}[section]
\newtheorem{lem}[thm]{Lemma}

\newtheorem{prop}[thm]{Proposition}
\newtheorem{cor}[thm]{Corollary}

\newtheorem{ass}[thm]{Assumption}
\theoremstyle{definition}
\newtheorem{defn}[thm]{Definition}

\newtheorem{rem}[thm]{Remark}

\title[Higher codimension Iwasawa theory]
{Higher codimension Iwasawa theory for elliptic curves with supersingular reduction}
\author[T. Kataoka]{Takenori Kataoka}
\address{Faculty of Science and Technology, Keio University.
3-14-1 Hiyoshi, Kohoku-ku, Yokohama, Kanagawa 223-8522, Japan}
\email{tkataoka@math.keio.ac.jp}
\date{}

\begin{document}

\begin{abstract}
Bleher et al.~began studying higher codimension Iwasawa theory for classical Iwasawa modules.
Subsequently, Lei and Palvannan studied an analogue for elliptic curves with supersingular reduction.
In this paper, we obtain a vast generalization of the work of Lei and Palvannan.
A key technique is an approach to the work of Bleher et al.~that the author previously proposed.
For this purpose, we also study the structure of $\pm$-norm subgroups and duality properties of multiply-signed Selmer groups.
\end{abstract}

\maketitle

\section{Introduction}\label{sec:01}

In Iwasawa theory, we study various Iwasawa modules, or Selmer groups, associated to various arithmetic objects.
In \cite{BCG+}, Bleher et al.~began studying {\it higher codimension} behavior of unramified Iwasawa modules which are assumed to be pseudo-null.
The pseudo-nullity of the unramified Iwasawa modules is known as Greenberg's conjecture.
In \cite{BCG+} they mainly deal with rank one cases in a sense, 
and subsequently in \cite{BCG+b} they extend the study to higher rank cases, where the unramified Iwasawa modules concerned are replaced by certain modules defined via exterior powers.

In \cite{Kata_12}, the author proposed a new approach to the theory of \cite{BCG+} and \cite{BCG+b}.
The new approach enables us to deal with equivariant situations and, moreover, to avoid localizing at prime ideals of height $2$.

On the other hand, Lei and Palvannan \cite{LP19} developed an analogue of the work \cite{BCG+}, concerning Selmer groups of elliptic curves.
This is partly motivated by a conjecture on the pseudo-nullity of the fine Selmer groups, which was predicted by Coates and Sujatha \cite[Conjecture B]{CS05} as an analogue of Greenberg's conjecture.
More concretely, given an elliptic curve with supersingular reduction over an imaginary quadratic field in which $p$ splits, Lei and Palvannan studied the doubly-signed Selmer groups.
Those Selmer groups were introduced by B.~D.~Kim \cite{Kim14} after the $\pm$-theory developed by Kobayashi \cite{Kob03}.

In the present paper, by applying the approach of \cite{Kata_12}, we develop higher codimension Iwasawa theory for elliptic curves with supersingular reduction.
This work generalizes the results of \cite{LP19} in the following aspects:
\begin{itemize}
\item
We work over a general base number field, whereas in \cite{LP19} the base field is an imaginary quadratic field.
\item
We deal with an arbitrary (even equivariant) abelian $p$-adic Lie extension of the base field (containing the cyclotomic $\Z_p$-extension), whereas \cite{LP19} deals with the unique $\Z_p^2$-extension of the imaginary quadratic field.
\item
We do not have to localize at prime ideals, whereas \cite{LP19} studies only behavior after localization at prime ideals of height $2$.
\item
We deal with arbitrary ranks, whereas \cite{LP19} deals with rank one cases only.
Here, the rank means the integer $l \geq 1$ in \S \ref{subsec:main1} below.
To do this, we make use of exterior powers, following the idea of \cite{BCG+b}.
\end{itemize}

Note that \cite{LP19} also deals with another kind of Selmer groups (defined by Greenberg's Panchishkin condition), but we do not study them in this paper.

The basic idea of this paper is the same as in \cite{Kata_12}.
However, we need a couple of extra ingredients which are specific to elliptic curves with supersingular reduction.
One is precise descriptions of the $\pm$-norm subgroups.
This will be studied in \S \ref{sec:27}, which relies on previous work \cite{Kata_08} of the author.
The results are of independent interest and are so precise that we can recover several previous results (see Remark \ref{rem:LL}).
Another is the behavior of Selmer groups under duality that we state as Theorem \ref{thm:E1}.
The result can be regarded as an extension of algebraic functional equations for multiply-signed Selmer groups, which is also of independent interest (see Remark \ref{rem:alg_FE}).

After introducing basic notations in \S \ref{subsec:notation}, we state the main results in \S \S \ref{subsec:main1} and \ref{subsec:main2}.

\subsection{Notations}\label{subsec:notation}

Let us fix an odd prime number $p$ and a number field $F$.
Let $E$ be an elliptic curve over $F$ that has good reduction at all $p$-adic primes of $F$.

Let $K_{\infty}/F$ be an abelian $p$-adic Lie extension.
Equivalently, $K_{\infty}/F$ is an abelian extension which is a finite extension of a multiple $\Z_p$-extension.
We suppose that $K_{\infty} \supset \mu_{p^{\infty}}$, where $\mu_{p^n}$ denotes the group of $p^n$-th roots of unity and $\mu_{p^{\infty}} = \bigcup_n \mu_{p^n}$.
The associated Iwasawa algebra is denoted by $\RR = \Z_p[[\Gal(K_{\infty}/F)]]$.
In this introduction we always assume Assumptions \ref{ass:reduction} and \ref{ass:non_anom} introduced in \S \ref{subsec:pm_subgroup}.

We write $S_p^{\ssr}$ for the set of $p$-adic primes of $F$ at which $E$ has supersingular reduction.
We call an element of $\prod_{\pe \in S_p^{\ssr}} \{+, -\}$ a {\it multi-sign}.
More generally, we call an element of $\prod_{\pe \in S_p^{\ssr}} \{0, 1, +, -, \rel\}$ a {\it multi-index} (here, $0$ and $1$ are just symbols and have no relation with the natural numbers).

For each multi-index $\epsilon = (\epsilon_{\pe})_{\pe}$, in \S \ref{subsec:Sel_defn}, we will introduce the $\epsilon$-Selmer group $\Sel^{\epsilon}(E/K_{\infty})$.
It is defined by imposing $\epsilon_{\pe}$-local condition at each $\pe \in S_p^{\ssr}$; $0$ denotes the strict condition, $\rel$ denotes the relaxed condition, and the $\pm$-local condition is essentially introduced by Kobayashi \cite{Kob03}.
The original work of Kobayashi dealt with the case where $F = \Q$, and B.~D.~Kim \cite{Kim14} gave an extension to the case where $F$ is an imaginary quadratic field.
The definition for general $F$ was given by Lei and Lim \cite{LL21}.

We take a finite set $S$ of non-$p$-adic finite primes of $F$ that satisfies a certain condition that we label as \eqref{eq:choiceS}.
Then we also define the $S$-imprimitive $\epsilon$-Selmer group $\Sel^{\epsilon}_{S}(E/K_{\infty})$ by relaxing the local condition at the primes in $S$.
We put 
\[
\dSel{\epsilon}{} = \Sel^{\epsilon}(E/K_{\infty})^{\vee},
\qquad \dSel{\epsilon}{S} = \Sel^{\epsilon}_{S}(E/K_{\infty})^{\vee},
\]
where $(-)^{\vee}$ denotes the Pontryagin dual.
It is known that $\dSel{\epsilon}{}$ and $\dSel{\epsilon}{S}$ are finitely generated $\RR$-modules.
Moreover, we assume Assumption \ref{ass:gen_tor}, which claims that these modules are torsion if $\epsilon$ is a multi-sign.

For each multi-sign $\epsilon$,
in \S \ref{subsec:pL_defn}, we will define an algebraic ($S$-imprimitive) $\epsilon$-$p$-adic $L$-function
\[
\LL_S^{\epsilon} = \LL_{S}^{\epsilon}(E/K_{\infty}) \in \RR
\]
by requiring $\Fitt_{\RR}(\dSel{\epsilon}{S}) = (\LL_{S}^{\epsilon})$, where $\Fitt_{\RR}(-)$ denotes the initial Fitting ideal.
Note that $\LL_S^{\epsilon}$ is defined in an algebraic way, not an analytic way.
Therefore, a main conjecture is expected to be formulated as an equality between $\LL_S^{\epsilon}$ and a certain analytic $p$-adic $L$-function, up to unit (see Remark \ref{rem:41} for the case $F = \Q$).
However, in this paper we do not study such a main conjecture and we only deal with algebraic aspects.

\subsection{The first main result}\label{subsec:main1}

Let us take distinct multi-signs $\epsilon_1, \dots, \epsilon_n \in \prod_{\pe \in S_p^{\ssr}} \{+, -\}$ with $n \geq 2$ (this forces $S_p^{\ssr} \neq \emptyset$).
We then define multi-indices $\ol{\epsilon}$ and $\ul{\epsilon}$ by
\[
(\ol{\epsilon}_{\pe}, \ul{\epsilon}_{\pe}) = 
\begin{cases}
(\rel, 1) & (\text{unless $\epsilon_{1, \pe} = \cdots = \epsilon_{n, \pe}$})\\
(\pm, \pm) & (\text{if $\epsilon_{1, \pe} = \cdots = \epsilon_{n, \pe} = \pm$})
\end{cases}
\] 
for each $\pe \in S_p^{\ssr}$.
Put 
\[
l = \# \{\pe \in S_p^{\ssr} \mid \ol{\epsilon}_{\pe} = \rel \} \geq 1.
\]

For each finite prime $v$ of $F$, we put $K_{\infty, v} = K_{\infty} \otimes_F F_v$.
For $1 \leq i \leq n$, we define
\[
\fD_i 
= \bigoplus_{\pe \in S_p^{\ssr}, \ol{\epsilon}_{\pe} = \rel} \parenth{H^1(K_{\infty, \pe}, E[p^{\infty}]) \over E^{\epsilon_{i, \pe}}(K_{\infty, \pe}) \otimes (\Q_p/\Z_p)}^{\vee},
\]
where $E^{\pm}(K_{\infty, \pe})$ denotes the $\pm$-norm subgroup of $E(K_{\infty, \pe})$ (see \S \ref{subsec:pm_subgroup}) and we use the Kummer map to embed $E^{\pm}(K_{\infty, \pe}) \otimes (\Q_p/\Z_p)$ into $H^1(K_{\infty, \pe}, E[p^{\infty}])$.
By comparing the local conditions between $\dSel{\ul{\epsilon}}{S}$ and $\dSel{\ol{\epsilon}}{S}$, we obtain a natural exact sequence of $\RR$-modules
\begin{equation}\label{eq:DSS}
\bigoplus_{i=1}^n \fD_i
\to \dSel{\ol{\epsilon}}{S}
\to \dSel{\ul{\epsilon}}{S}
\to 0.
\end{equation}
As we will see in \S \ref{subsec:CM_15}, the generic ranks of both $\fD_i$ and $\dSel{\ol{\epsilon}}{S}$ as $\RR$-modules are $l$.
It is of critical importance that $\fD_i$ is {\it free} as an $\RR$-module.
To prove this fact, in \S \ref{sec:27} we study the structure of $E^{\pm}(K_{\infty, \pe})$ closely.

Then \eqref{eq:DSS} implies that the first map to $\dSel{\ol{\epsilon}}{S}$ has information about $\dSel{\ul{\epsilon}}{S}$.
Motivated by the work \cite{BCG+b}, we take the exterior powers and obtain a map
\[
\bigoplus_{i=1}^n \bigwedge_{\RR}^l \fD_i
\to \bigwedge_{\RR}^l \dSel{\ol{\epsilon}}{S}.
\]
The module in Theorem \ref{thm:main} below involving exterior powers denotes the cokernel of this map after taking the quotient of the target module by its torsion part.

In general, for a finitely generated $\RR$-module $M$ and $i \geq 0$, we put $E^i(M) = \Ext_{\RR}^i(M, \RR)$.
We define $M_{\tor}$ (resp.~$M_{\PN}$) as the maximal torsion (resp.~pseudo-null) submodule of $M$ and we put $M_{/\tor} = M/M_{\tor}$ (resp.~$M_{/\PN} = M / M_{\PN}$).

\begin{thm}\label{thm:main}
We have an exact sequence
\begin{align}
0 \to \frac{\left(\bigwedge_{\RR}^l \dSel{\ol{\epsilon}}{S} \right)_{/\tor}}
{\sum_{i=1}^n \bigwedge_{\RR}^l \fD_i}
\to \frac{\RR}{\sum_{i=1}^n(\LL_{S}^{\epsilon_i})}
\to \frac{\RR}{\Fitt_{\RR} (E^1(\dSel{\ol{\epsilon}}{S}) )}
\to 0.
\end{align}
\end{thm}

The proof of Theorem \ref{thm:main} will be given in \S \ref{subsec:CM_15}.
The structure of $E^1(\dSel{\ol{\epsilon}}{S})$ will be the theme of Theorem \ref{thm:E1} below.
Actually, by combining Theorems \ref{thm:main} and \ref{thm:E1}, we obtain an analogue of \cite[Theorem 5.3]{Kata_12} for classical Iwasawa modules for CM-fields.
(The statement of the results in this paper is much simpler than that of \cite{Kata_12}; this is essentially because we have $H^0(K_{\infty, \pe}, E[p^{\infty}]) = 0$ for any $p$-adic prime $\pe$ of $F$.)

We also have a refined version of Theorem \ref{thm:main} for the $l = 1$ case.
Note that $l = 1$ is equivalent to that $n = 2$ and the two multi-signs $\epsilon_1$ and $\epsilon_2$ differ at a single component.
By assuming $l = 1$, thanks to \eqref{eq:DSS}, we can immediately reformulate Theorem \ref{thm:main} as claim (1) of the following corollary; claim (2) will be proved in \S \ref{subsec:pf_l=1}.

\begin{cor}\label{cor:l=1}
Let us suppose $l = 1$.

(1)
We have an exact sequence
\begin{align}
\left(\dSel{\ol{\epsilon}}{S} \right)_{\tor}
\to \dSel{\ul{\epsilon}}{S}
\to \frac{\RR}{(\LL_{S}^{\epsilon_1}, \LL_{S}^{\epsilon_2})}
 \to \frac{\RR}{\Fitt_{\RR} \left(E^1(\dSel{\ol{\epsilon}}{S}) \right)}
\to 0.
\end{align}

(2)
Let us moreover suppose that $E^1(\dSel{\ol{\epsilon}}{S})$ is pseudo-null over $\RR$.
Then we have $\left(\dSel{\ol{\epsilon}}{S} \right)_{\tor} = 0$ and an (abstract) isomorphism
\[
\frac{\RR}{\Fitt_{\RR} \left(E^1(\dSel{\ol{\epsilon}}{S}) \right)}
\simeq E^2(E^1(\dSel{\ol{\epsilon}}{S})).
\]
Therefore, we obtain an exact sequence
\[
0 \to \dSel{\ul{\epsilon}}{S}
\to \frac{\RR}{(\LL_{S}^{\epsilon_1}, \LL_{S}^{\epsilon_2})}
\to E^2(E^1(\dSel{\ol{\epsilon}}{S}))
\to 0.
\]
\end{cor}

Note that the pseudo-nullity of $E^1(\dSel{\ol{\epsilon}}{S})$ is closely related to that of $\dSel{\eta}{}$ in Theorem \ref{thm:E1} below (as long as we assume condition ($\star$) there).
Taking \cite[Conjecture B]{CS05} of Coates and Sujatha into account, we may expect that $\dSel{\eta}{}$ is often pseudo-null.

\subsection{The second main result}\label{subsec:main2}

The following is the second main result of this paper (the proof will be given in \S \ref{sec:pf_E1}).
It gives an alternative description of $E^1(\dSel{\ol{\epsilon}}{S})$ in Theorem \ref{thm:main}; we set $\epsilon$ as $\ol{\epsilon}$ in Theorem \ref{thm:main}.

\begin{thm}\label{thm:E1}
We take a multi-index $\epsilon = (\epsilon_{\pe})_{\pe} \in \prod_{\pe \in S_p^{\ssr}} \{+, -, \rel\}$.
We define a multi-index $\eta = (\eta_{\pe})_{\pe} \in \prod_{\pe \in S_p^{\ssr}} \{0, +, -\}$ by
\[
\eta_{\pe} = 
\begin{cases}
	\epsilon_{\pe} & (\text{if $\epsilon_{\pe} \in \{+, -\}$})\\
	0 & (\text{if $\epsilon_{\pe} = \rel$}).
\end{cases}
\]
We suppose:
\begin{itemize}
\item[$(\star)$]
For each $\pe \in S_p^{\ssr}$ with $\epsilon_{\pe} = +$, the residue degree of $K_{\infty}/F$ at $\pe$ is not divisible by $4$.
\end{itemize}
Then we have an exact sequence
\begin{align}
0 \to (\dSel{\eta}{})^{\iota}
 \to E^1(\dSel{\epsilon}{S})
\to \bigoplus_{v \in S} E(K_{\infty, v})[p^{\infty}]^{\vee, \iota}
\oplus \bigoplus_{v \not\in S_p(F) \cup S} E^2 \parenth{\parenth{E(K_{\infty, v})[p^{\infty}]^{\vee}}_{\PN}}
\to 0.
\end{align}
\end{thm}

Here, $\bigoplus_{v \not\in S_p(F) \cup S}$ denotes the direct sum for the finite primes $v$ of $F$ with $v \not \in S_p(F) \cup S$ (see Remark \ref{rem:bad} below).
We write $\iota$ for the involution on $\RR$ that inverts every group element.
For an $\RR$-module $M$, we write $M^{\iota}$ for the $\RR$-module whose additive structure is the same as $M$ and the action of $\RR$ is twisted by $\iota$.

Note that, in condition $(\star)$, the residue degree is not finite in general, and in that case we regard it as a supernatural number (cf.~Definition \ref{defn:supernatural}).
Without the assumption $(\star)$, the description of $E^1(\dSel{\epsilon}{S})$ seems to get harder.

\begin{rem}\label{rem:bad}
For $v \not \in S_p(F)$, it is not hard to see $\parenth{E(K_{\infty, v})[p^{\infty}]^{\vee}}_{\PN} = 0$ if $E$ has good reduction at $v$.
Therefore, in the last direct sum of Theorem \ref{thm:E1}, we may restrict the range of $v$ to the primes at which $E$ has bad reduction.
In the situation of \cite{LP19}, the corresponding factors already played a role in \cite[Theorem 1]{LP19}, and the structure was studied in \cite[\S 7.3]{LP19}.
However, only the second Chern classes were computed in that work.
The exact sequence in Theorem \ref{thm:E1} is a novel observation of this paper.
\end{rem}

\begin{rem}\label{rem:alg_FE}
In the main stream of this paper, we apply Theorem \ref{thm:E1} to the case where $\epsilon$ is $\ol{\epsilon}$ in Theorem \ref{thm:main}, which is not a multi-sign but a multi-index.
On the other hand, we can also apply Theorem \ref{thm:E1} to multi-signs $\epsilon$.
Note that then $\eta = \epsilon$ by the definition.
In this case, Theorem \ref{thm:E1} can be regarded as an algebraic functional equation for the multiply-signed Selmer groups.
The result is a refinement of previous work (e.g., Ahmed and Lim \cite[Theorem 3.3]{AL21}).
Actually, the previous work mainly focused on the pseudo-isomorphism classes, and the exact sequence in Theorem \ref{thm:E1} provides us more information.
To recover the previous results (in non-equivariant settings), we only need to consider $S = \emptyset$ and observe that $E^1(-)$ does not change the pseudo-isomorphism classes for finitely generated torsion $\RR$-modules.
Note also that our proof of Theorem \ref{thm:E1} relies on a complex version of duality (Proposition \ref{prop:cpx_dual}), which is even stronger.
\end{rem}

\subsection{Organization of this paper}\label{subsec:outline}

In \S \ref{sec:Sel_p-L}, we give the definitions of Selmer groups and algebraic $p$-adic $L$-functions.
In \S \ref{sec:app}, we illustrate the main results of this paper in special cases; in particular, we explain how to recover a main result of \cite{LP19}.
In \S \ref{sec:27}, we study the structures of $\pm$-norm subgroups.
In \S \S \ref{sec:02} and \ref{sec:CM_14}, we review facts on perfect complexes and then introduce arithmetic complexes whose cohomology groups know the Selmer groups concerned.
In \S \S \ref{sec:pf_main} and \ref{sec:pf_E1}, we prove the first and the second main results, respectively.

\section{Definitions of Selmer groups and algebraic $p$-adic $L$-functions}\label{sec:Sel_p-L}

In this section, we give the definitions of the Selmer groups and the algebraic $p$-adic $L$-functions.
We keep the notations in \S \ref{subsec:notation} and here introduce some more notations.

Let $S_p(F)$ (resp.~$S_{\infty}(F)$) denote the set of $p$-adic primes (resp.~archimedean places) of $F$.
As in \cite[\S 3.1]{Kata_12}, we define $S_{\ram, p}(K_{\infty}/F)$ as the set of finite primes $v$ of $F$ such that $v$ is not lying above $p$ and that the ramification index of $K_{\infty}/F$ at $v$ is divisible by $p$.
For instance, we have $S_{\ram, p}(K_{\infty}/F) = \emptyset$ as long as $\Gal(K_{\infty}/F)$ does not contain an element of order $p$ (this case will be called the non-equivariant case).
The set $S_{\ram, p}(K_{\infty}/F)$ corresponds to the set $\Phi_{K/F}$ in \cite[Theorem 1]{Gree11}, and the necessity of the set in our study is also explained in \cite[Proposition 3.1]{Kata_12}.
We take a finite set $S$ of primes of $F$ such that
\begin{equation}\label{eq:choiceS}
S \supset S_{\ram, p}(K_{\infty}/F), 
\qquad S \cap (S_p(F) \cup S_{\infty}(F)) = \emptyset.
\end{equation}

For each finite prime $v$ of $F$, let $F_v$ be the localization of $F$ at $v$ and we put $K_{\infty, v} = K_{\infty} \otimes_F F_v$.
Note that then $K_{\infty, v}$ is the inductive limit of $K' \otimes_F F_v$, where $K'$ runs over intermediate number field in $K_{\infty}/F$ and so each $K' \otimes_F F_v$ is a finite product of fields.
In general we interpret cohomology group $H^i(K_{\infty, v}, -)$ as the inductive limit of $H^i(K' \otimes_F F_v, -)$.

We define $S_p^{\ord}$ as the set of $p$-adic primes of $F$ at which $E$ has ordinary reduction.
Note that then $S_p(F)$ is the disjoint union of $S_p^{\ssr}$ and $S_p^{\ord}$.
We are mainly interested in the case where $S_p^{\ord} = \emptyset$ and (equivalently) $S_p^{\ssr} = S_p(F)$, but we do not assume this for more generality.

\subsection{The $\pm$-norm subgroups}\label{subsec:pm_subgroup}

As already remarked, the key idea to define the signed Selmer groups is given by Kobayashi \cite{Kob03}, and there are a number of subsequent studies to generalize the idea.
The definition below basically follows \cite[Definition 4.7]{LL21}.

As usual, for each prime $\pe \in S_p(F)$, we put $a_{\pe}(E) = (1 + \# \F_{\pe}) - \# \ttilde{E}(\F_{\pe})$, where $\F_{\pe}$ denotes the residue field of $F$ at $\pe$ and $\ttilde{E}$ denotes the reduction of $E$ at $\pe$.
We also put $\deg(\pe) = [F_{\pe}:\Q_p]$.

In order to use the $\pm$-theory, we need to assume the following.

\begin{ass}\label{ass:reduction}
For each $\pe \in S_p^{\ssr}$, we have $\deg(\pe) = 1$ and $a_{\pe}(E) = 0$.
\end{ass}

We also assume the non-anomalous condition at ordinary primes:

\begin{ass}\label{ass:non_anom}
For each $\pe \in S_p^{\ord}$, we have $E(K_{\infty, \pe})[p] = 0$.
\end{ass}

Let $\pe \in S_p^{\ssr}$.
We define $M^{\pe}$ as the inertia field in the given extension $K_{\infty}/F$ at $\pe$.
Thanks to Assumption \ref{ass:reduction} and $K_{\infty} \supset \mu_{p^{\infty}}$, it is easy to see $K_{\infty} = M^{\pe}(\mu_{p^{\infty}})$ and $\Gal(K_{\infty}/M^{\pe}) \simeq \Z_p^{\times}$.
For each integer $n \geq -1$, we put $M^{\pe}_n = M^{\pe}(\mu_{p^{n+1}})$.

\begin{defn}\label{defn:CM_31}
Let $\pe \in S_p^{\ssr}$.
For each $n \geq -1$ and for each choice of $\pm$, we define
\begin{align}
& E^{\pm}(M^{\pe}_n \otimes F_{\pe})\\
& = \{ x \in E(M^{\pe}_n \otimes F_{\pe}) \mid \Tr_{M^{\pe}_n/M^{\pe}_{n'+1}}(x) \in E(M^{\pe}_{n'} \otimes F_{\pe}), -1 \leq \forall n' < n, (-1)^{n'} = \pm 1\}.
\end{align}
Here, $\Tr_{M^{\pe}_n/M^{\pe}_{n'+1}}$ denotes the trace map from $E(M^{\pe}_n \otimes F_{\pe})$ to $E(M^{\pe}_{n'+1} \otimes F_{\pe})$.
The tensor products are taken over $F$.
We define
\[
E^{\pm}(K_{\infty, \pe}) = \bigcup_{n \geq -1} E^{\pm}(M^{\pe}_n \otimes F_{\pe}).
\]
\end{defn}

\begin{defn}\label{defn:Loc}
Let $\pe \in S_p^{\ssr}$.
For $\bullet \in \{0, 1, +, -, \rel\}$, we define an $\RR$-submodule $\Loc_{\pe}^{\bullet}$ of $H^1(K_{\infty, \pe}, E[p^{\infty}])$ by
\[
\Loc_{\pe}^{\bullet} =
\begin{cases}
	0 & (\text{if $\bullet = 0$})\\
	E(M^{\pe} \otimes F_{\pe}) \otimes (\Q_p / \Z_p) & (\text{if $\bullet = 1$})\\
	E^{\bullet}(K_{\infty, \pe}) \otimes (\Q_p/\Z_p) & (\text{if $\bullet \in \{+, -\}$})\\
	H^1(K_{\infty, \pe}, E[p^{\infty}]) & (\text{if $\bullet = \rel$}).
\end{cases}
\]
Here, when $\bullet \in \{1, +, -\}$, we use the Kummer map to embed $\Loc_{\pe}^{\bullet}$ into $H^1(K_{\infty, \pe}, E[p^{\infty}])$ (it is actually injective).
Note that, by (a semi-local variant of) Proposition \ref{prop:pm_str} below, we actually have
$\Loc_{\pe}^{1} = \Loc_{\pe}^{+} \cap \Loc_{\pe}^{-}$.
This is a motivation for the definition of $\Loc_{\pe}^{1}$ (the symbol $1$ is used in \cite[\S 10.1]{Kob03}).
\end{defn}

\subsection{The Selmer groups}\label{subsec:Sel_defn}

\begin{defn}\label{defn:CM_35}
Let $\epsilon = (\epsilon_{\pe})_{\pe} \in \prod_{\pe \in S_p^{\ssr}} \{0, 1, +, -, \rel\}$ be a multi-index.
We define the $S$-imprimitive $\epsilon$-Selmer group $\Sel^{\epsilon}_{S}(E/K_{\infty})$ as the kernel of the localization homomorphism
\begin{align}
H^1(K_{\infty}, E[p^{\infty}]) 
\to  \bigoplus_{\pe \in S_p^{\ssr}} \frac{H^1(K_{\infty, \pe}, E[p^{\infty}])}{\Loc_{\pe}^{\epsilon_{\pe}}}
& \oplus \bigoplus_{\pe \in S_p^{\ord}} \frac{H^1(K_{\infty, \pe}, E[p^{\infty}])}{E(K_{\infty, \pe}) \otimes (\Q_p/\Z_p)}\\
& \oplus \bigoplus_{v \not \in S_p(F) \cup S} H^1(K_{\infty, v}, E[p^{\infty}]),
\end{align}
where, in the last direct sum, $v$ runs over the finite primes of $F$ not in $S_p(F) \cup S$.
We also define the (primitive) $\epsilon$-Selmer group $\Sel^{\epsilon}(E/K_{\infty})$ as the kernel of the localization homomorphism
\[
\Sel^{\epsilon}_{S}(E/K_{\infty})
\to \bigoplus_{v \in S} H^1(K_{\infty, v}, E[p^{\infty}]).
\]
\end{defn}

As in \S \ref{subsec:notation}, we write $\dSel{\epsilon}{}$ and $\dSel{\epsilon}{S}$ for the Pontryagin duals of $\Sel^{\epsilon}(E/K_{\infty})$ and $\Sel^{\epsilon}_{S}(E/K_{\infty})$, respectively.
As stated in \cite[Conjecture 4.11]{LL21}, it is natural to conjecture the following.

\begin{ass}\label{ass:gen_tor}
For each multi-sign $\epsilon \in \prod_{\pe \in S_p^{\ssr}} \{+, -\}$, the $\RR$-module 
$\dSel{\epsilon}{S}$ is torsion.
\end{ass}

When $F = \Q$, Assumption \ref{ass:gen_tor} is known to be true, thanks to the celebrated work \cite{Kato04} of Kato (see \cite[Theorem 1.2]{Kob03} or \cite[Proposition 2.3]{Kata_08}).
See the last paragraph of \cite[\S 6.2]{LP19} for progress on the case where $F$ is an imaginary quadratic field.

It is convenient to introduce an order on the 5-element set $\{0, 1, +, -, \rel\}$ defined by
\[
0 < 1, 
\qquad 1 < + < \rel, 
\qquad 1 < - < \rel
\]
(there is no order between $+$ and $-$).
We extend this order to the set $\prod_{\pe \in S_p^{\ssr}} \{0, 1, +, -, \rel\}$ by defining $\epsilon' \leq \epsilon$ if and only if $\epsilon'_{\pe} \leq \epsilon_{\pe}$ for every $\pe \in S_p^{\ssr}$.
Then we have $\Sel^{\epsilon'}_{S}(E/K_{\infty}) \subset \Sel^{\epsilon}_{S}(E/K_{\infty})$ if $\epsilon' \leq \epsilon$ since we have the corresponding inclusions concerning the local conditions.
Note also that the definition of $\ol{\epsilon}, \ul{\epsilon}$ in \S \ref{subsec:main1} can be rephrased as
\[
\ol{\epsilon} = \sup \{\epsilon_1, \dots, \epsilon_n\},
\quad \ul{\epsilon} = \inf \{\epsilon_1, \dots, \epsilon_n\}.
\]

\subsection{The $p$-adic $L$-functions}\label{subsec:pL_defn}

Next we define the algebraic $p$-adic $L$-functions.

\begin{defn}\label{defn:CM_37}
Let $\epsilon \in \prod_{\pe \in S_p^{\ssr}} \{+, -\}$ be a multi-sign such that Assumption \ref{ass:gen_tor} holds.
As we will show in Proposition \ref{prop:CM_cohom_L}, we have
\[
\pd_{\RR} \parenth{\dSel{\epsilon}{S}} \leq 1,
\]
where $\pd$ denotes the projective dimension.
(When $F = \Q$, this is nothing but \cite[Theorem 1.1 and Remark 5.9]{Kata_08}.
More generally this is also established by Lim \cite[Theorem 4.8]{Lim21} for the case where $K_{\infty}$ is the cyclotomic $\Z_p$-extension of a number field.)
Note that assumption \eqref{eq:choiceS} is required here.
Then we define the algebraic $\epsilon$-$p$-adic $L$-function $\LL_S^{\epsilon} = \LL_{S}^{\epsilon}(E/K_{\infty}) \in \RR$, up to units, by requiring
\[
\Fitt_{\RR}(\dSel{\epsilon}{S}) = (\LL_{S}^{\epsilon})
\]
as principal ideals of $\RR$.
\end{defn}

\begin{rem}\label{rem:41}
Let us suppose $F = \Q$ and write $\LL_{S}^{\an, \pm}$ for the analytic $p$-adic $L$-function (see \cite[\S 2.2]{Kata_08} for the precise definition).
Then, as in \cite[Equation (1.1)]{Kata_08}, the (equivariant) main conjecture should be formulated as an equality between principal ideals
\[
W^{\pm} (\LL_{S}^{\pm}) = (\LL_{S}^{\an, \pm}),
\]
where $W^{\pm}$ is an explicit auxiliary ideal.
\end{rem}

\section{Applications of the main results}\label{sec:app}

In this section,
we illustrate the main results of this paper in the case where $F$ is either $\Q$ or an imaginary quadratic field.

\subsection{The case where $F = \Q$}\label{subsec:F=Q}

We suppose $F = \Q$, so we consider an abelian extension $K_{\infty}/\Q$ which is a finite extension of $\Q(\mu_{p^{\infty}})$.
Let $E/\Q$ be an elliptic curve which has good supersingular reduction at $p$.
We moreover suppose $a_p(E) = 0$ (this is automatically true if $p \geq 5$ by the Hasse bound).
Since $S_p^{\ssr}$ is a singleton, a multi-index can be simply denoted by an element of $\{0, 1, +, -, \rel\}$.
Note that Assumptions \ref{ass:reduction}, \ref{ass:non_anom}, and \ref{ass:gen_tor} hold automatically.

The unique choice (up to permutation) of distinct multi-signs is $\epsilon_1 = +$ and $\epsilon_2 = -$.
Then we have $l = 1$ and $\ol{\epsilon} = \rel$, $\ul{\epsilon} = 1$.
As a consequence of Corollary \ref{cor:l=1} and Theorem \ref{thm:E1} (note that the condition $(\star)$ trivially holds), we obtain the following.

\begin{thm}\label{thm:F=Q}
The following are true.
\begin{itemize}
\item[(1)]
We have an exact sequence
\begin{align}
\parenth{\dSel{\rel}{S}}_{\tor}
\to \dSel{1}{S}
\to \frac{\RR}{(\LL_{S}^+, \LL_{S}^-)}
\to \frac{\RR}{\Fitt_{\RR} \parenth{E^1(\dSel{\rel}{S})}}
\to 0.
\end{align}
\item[(2)]
If $E^1(\dSel{\rel}{S})$ is pseudo-null over $\RR$ (i.e., is finite),
then we have an exact sequence
\[
0 \to \dSel{1}{S}
\to \frac{\RR}{(\LL_{S}^+, \LL_{S}^-)}
\to E^2(E^1(\dSel{\rel}{S}))
\to 0.
\]
\item[(3)]
We have an exact sequence
\begin{align}
0 \to (\dSel{0}{})^{\iota} 
 \to E^1(\dSel{\rel}{S})
\to \bigoplus_{v \in S} E(K_{\infty, v})[p^{\infty}]^{\vee, \iota}
\oplus \bigoplus_{v \not \in S_p(\Q) \cup S} E^2 \parenth{\parenth{E(K_{\infty, v})[p^{\infty}]^{\vee}}_{\PN}}
\to 0.
\end{align}
\end{itemize}
\end{thm}

\subsection{The case where $F$ is an imaginary quadratic field}\label{subsec:pf_LP}

We shall deduce a main result of \cite{LP19} from Corollary \ref{cor:l=1} and Theorem \ref{thm:E1}.
Let us suppose that $F$ is an imaginary quadratic field in which $p$ splits into $\pe$ and $\ol{\pe}$.
We also suppose that the dimension of $\Gal(K_{\infty}/F)$ is two.
Equivalently, $K_{\infty}/F$ is an abelian extension which is a finite extension of $\ttilde{F}(\mu_p)$, where $\ttilde{F}$ denotes the unique $\Z_p^2$-extension of $F$.

Let $E/F$ be an elliptic curve which has good supersingular reduction at both $\pe$ and $\ol{\pe}$, and we suppose that $a_{\pe}(E) = a_{\ol{\pe}}(E) = 0$ so that Assumptions \ref{ass:reduction} and \ref{ass:non_anom} hold.
Fixing the order $\pe, \ol{\pe}$, we express a multi-index $\epsilon$ simply by writing $(\epsilon_{\pe}, \epsilon_{\ol{\pe}})$.

Then there are, up to permutation, exactly 4 choices of $\epsilon_1, \epsilon_2$ such that $l = 1$, namely
\begin{equation}\label{eq:choiceE}
(\epsilon_1, \epsilon_2) = ((+, +), (+, -)), ((-, +), (-, -)), ((+, +), (-, +)), ((+, -), (-, -)).
\end{equation}
Moreover, we define $\eta$ as in Theorem \ref{thm:E1} for $\ol{\epsilon}$;
in other words, $\eta$ equals $\ul{\epsilon}$ with $1$ replaced by $0$.
For instance, when $\epsilon_1 = (+, +)$ and $\epsilon_2= (+, -)$, we have $\ol{\epsilon} = (+, \rel)$, $\ul{\epsilon} = (+, 1)$, and $\eta = (+, 0)$.

As a consequence of Corollary \ref{cor:l=1} and Theorem \ref{thm:E1}, we obtain a completely analogous theorem to Theorem \ref{thm:F=Q}.
We omit to state it.
Note that we have to assume the validity of Assumption \ref{ass:gen_tor}, and in addition condition ($\star$) to apply Theorem \ref{thm:E1}.

Our theorem recovers \cite[Theorem 1]{LP19} (except for the cases involving $\theta^{\textrm{Gr}}_{4, 2}$) in the following way.
We consider $K_{\infty} = \ttilde{F}(\mu_p)$.
Then $S_{\ram, p} = \emptyset$, so $S = \emptyset$ satisfies \eqref{eq:choiceS}.
We write $\LL^{\epsilon} = \LL_{\emptyset}^{\epsilon}$ for each multi-sign $\epsilon$.
Note that $\RR$ is a finite product of regular local rings.
For a pseudo-null $\RR$-module $M$, as in \cite[\S 1.1]{BCG+}, we define the second Chern class $c_2(M)$ by a formal sum
\[
c_2(M) = \sum_{\qu} \length_{\RR_{\qu}}(M_{\qu}) [\qu],
\]
where $\qu$ runs over the prime ideals of $\RR$ of height $2$.

\begin{cor}[{\cite[Theorem 1]{LP19}}]\label{cor:LP}
Let $(\epsilon_1, \epsilon_2)$ be one of \eqref{eq:choiceE}.
Let $K_{\infty} = \ttilde{F}(\mu_p)$ and suppose Assumption \ref{ass:gen_tor} for $\epsilon_1, \epsilon_2$.
We moreover suppose that the elements $\LL^{\epsilon_1}$ and $\LL^{\epsilon_2}$ of $\RR$ are coprime to each other.
Then we have
\begin{align}
c_2 \parenth{\dSel{\ul{\epsilon}}{}}
+ c_2 \parenth{(\dSel{\eta}{})^{\iota}}
+ \sum_{v \not \in S_p(F)} c_2 \parenth{ \parenth{E(K_{\infty, v})[p^{\infty}]^{\vee}}_{\PN}}
= c_2 \parenth{\frac{\RR}{(\LL^{\epsilon_1}, \LL^{\epsilon_2})}}.
\end{align}
\end{cor}

\begin{proof}
By the assumption and Corollary \ref{cor:l=1}(1), the module $E^1(\dSel{+, \rel}{})$ is pseudo-null.
Moreover, condition ($\star$) holds as $K_{\infty} = \ttilde{F}(\mu_p)$.
Therefore, we can apply Corollary \ref{cor:l=1}(2) and Theorem \ref{thm:E1}.
Then we only have to use the additivity of $c_2(-)$ with respect to exact sequences, together with the fact \cite[Remark 5.10]{BCG+b} that $c_2(E^2(M)) = c_2(M)$ for each pseudo-null module $M$.
\end{proof}

\section{Structures of the local conditions}\label{sec:27}

In this section, we study the local conditions for elliptic curves.
In \S \S \ref{subsec:local_settings}--\ref{subsec:conseq_str}, we study the $\pm$-norm subgroups for supersingular elliptic curves over $\Q_p$.
The results are very close to \cite[Theorem 1.2(3)]{Kata_08}, and actually the basic strategy is the same.
However, we have to generalize the situation from {\it finite} unramified extensions of $\Q_p$ to {\it infinite} unramified extensions of $\Q_p$.
We will accomplish the task by taking the limit suitably.
See Remark \ref{rem:LL} for a relation with other previous work (Lei and Lim \cite{LL21} and Lim \cite{Lim21}).
Finally in \S \ref{subsec:ord_local}, we briefly observe ordinary analogues.

\subsection{The local situation}\label{subsec:local_settings}

It is convenient to introduce the following formal terminology.

\begin{defn}\label{defn:supernatural}
Let $\bN = \{0, 1, 2, \dots\}$ be the set of nonnegative integers.
We put $\bN^* = \bN \cup \{\infty\}$.
An element of $\prod_{l} \bN^*$, where $l$ runs over all prime numbers, is called a {\it supernatural number}.
If $g = (g_l)_l$ is a supernatural number, we also write $g_l = \ord_{l}(g) \in \bN^*$ and we express $g$ as a formal product $g = \prod_l l^{\ord_l(g)}$.
This notion is an extension of positive integers; a positive integer $f$ can be decomposed uniquely as $f = \prod_{l} l^{\ord_l(f)}$, where $\ord_l(f) \in \bN$ is the normalized additive valuation of $f$ at $l$.
For supernatural numbers $g$ and $g'$, we write $g \mid g'$ if $\ord_l(g) \leq \ord_l(g')$ holds for all prime numbers $l$.
\end{defn}

In this section, we basically write $f$ for a positive integer and $g$ for a supernatural number.
Note that, for the applications to global settings, we only need to consider supernatural numbers $g$ of the form $g = f$ or $g = f p^{\infty}$ with positive integers $f$.

For each positive integer $f$, let $\F_{p^f}$ be the finite field with $p^f$ elements (in a fixed algebraic closure $\ol{\F_p}$ of $\F_p$), and $\Q_{p^f}$ the unramified extension of $\Q_p$ of degree $f$ (in a fixed maximal unramified extension $\Q_p^{\ur}$ of $\Q_p$).
For each supernatural number $g$, we put 
\[
\F_{p^g} = \bigcup_{f \mid g} \F_{p^f}, 
\qquad \Q_{p^g} = \bigcup_{f \mid g} \Q_{p^f},
\]
 where $f$ runs over the positive integers with $f \mid g$.
Then the correspondence $g \mapsto \F_{p^g}$ (resp.~$g \mapsto \Q_{p^g}$) is a bijection between the set of supernatural numbers and the set of intermediate fields of $\ol{\F_p}/\F_p$ (resp.~of $\Q_p^{\ur}/\Q_p$).
We write $\varphi \in \Gal(\Q_p^{\ur}/\Q_p)$ for the arithmetic Frobenius.

Let $g$ be a supernatural number.
For an integer $n \geq -1$, we put $\Q_{p^g, n} = \Q_{p^g}(\mu_{p^{n+1}})$ and $R_{g, n} = \Z_p[[\Gal(\Q_{p^g, n}/\Q_p)]]$.
We also put $\Q_{p^g, \infty} = \Q_{p^g}(\mu_{p^{\infty}})$ and $\RR_{g} = \Z_p[[\Gal(\Q_{p^g, \infty}/\Q_p)]]$.
Since $\Gal(\Q_{p^g, 0}/\Q_{p^g, -1}) \simeq \Gal(\Q_p(\mu_p)/\Q_p)$ is of order $p-1$, we can decompose the algebra $R_{g, 0}$ with respect to the characters of that Galois group.
Noting that the trivial character component is isomorphic to $R_{g, -1}$, we write $R_{g, 0}^{\nt}$ for the direct product of the non-trivial character components, so we have a natural decomposition as an algebra
\[
R_{g, 0} \simeq R_{g, -1} \times R_{g, 0}^{\nt}.
\]

\subsection{Structures of $\pm$-norm subgroups}\label{subsec:local_state}

Let $E$ be a supersingular elliptic curve over $\Q_p$ satisfying $a_p(E) = (1 + p) - \# \ttilde{E}(\F_p) = 0$.

We first define the $\pm$-norm subgroups in the current local situation.

\begin{defn}\label{defn:76}
Let $g$ be a supernatural number.
For each $n \geq -1$ and a choice of $\pm$, we define
\[
E^{\pm}(\Q_{p^g, n}) 
= \{ x \in E(\Q_{p^g, n}) \mid \Tr_{\Q_{p^g, n}/\Q_{p^g, n'+1}}(x) \in E(\Q_{p^g, n'}), -1 \leq \forall n' < n, (-1)^{n'} = \pm 1\}.
\]
We also define
\[
E^{\pm}(\Q_{p^g, \infty}) = \bigcup_{n \geq -1} E^{\pm}(\Q_{p^g, n}). 
\]
\end{defn}

Note that we have $E^{\pm}(\Q_{p^g, n}) = \bigcup_{f \mid g} E^{\pm}(\Q_{p^f, n})$, where $f$ runs over all positive integers with $f \mid g$.

For a positive integer $f$, we define $E(\Q_{p^f, n})_p$ as the $p$-adic completion of $E(\Q_{p^f, n})$, which we regard as a submodule of $E(\Q_{p^f, n})$ of finite index (the index is prime to $p$).
We also define $E^{\pm}(\Q_{p^f, n})_p$ similarly.
Then, for each supernatural number $g$, we define $E(\Q_{p^g, n})_p$ and $E^{\pm}(\Q_{p^g, n})_p$ as the union of $E(\Q_{p^f, n})_p$ and $E^{\pm}(\Q_{p^f, n})_p$ for $f \mid g$, respectively.

\begin{prop}\label{prop:pm_str}
Let $g$ be a supernatural number.
Let $n$ be an integer $\geq -1$ or $n = \infty$.
Then we have 
\[
E^+(\Q_{p^g, n})_p \cap E^-(\Q_{p^g, n})_p = E(\Q_{p^g, -1})_p,
\qquad E^+(\Q_{p^g, n})_p + E^-(\Q_{p^g, n})_p = E(\Q_{p^g, n})_p.
\]
\end{prop}

\begin{proof}
When $g = f$ is a positive integer, this is known (e.g., \cite[Proposition 3.16]{KO18}).
Note that the proof makes use of a family of local points that we review in Proposition \ref{prop:80} below.
Then the general case follows from taking the inductive limit with respect to $f \mid g$.
\end{proof}

The goal of this subsection is to prove the following (cf.~\cite[Theorem 1.2(3)]{Kata_08}).

\begin{prop}\label{prop:83}
Let $g$ be a supernatural number.
\begin{itemize}
\item[(1)]
We have exact sequences of $\RR_g$-modules
\[
0 \to (E^+(\Q_{p^g, \infty}) \otimes (\Q_p/\Z_p))^{\vee} 
\to \RR_{g} \oplus R_{g, -1} \overset{\ast}{\to} R_{g, -1} \to 0,
\]
where the first component of $\ast$ is the natural projection and the second is given by $\varphi + \varphi^{-1}$,
and
\[
0 \to (E^-(\Q_{p^g, \infty}) \otimes (\Q_p/\Z_p))^{\vee} 
\to \RR_{g} \overset{\ast}{\to} R_{g, 0}^{\nt} \to 0,
\]
where $\ast$ is the natural projection.
\item[(2)]
In particular, the $\RR_g$-module $(E^{\pm}(\Q_{p^g, \infty}) \otimes (\Q_p/\Z_p))^{\vee}$ is generically of rank one and satisfies
\[
\pd_{\RR_g}((E^{\pm}(\Q_{p^g, \infty}) \otimes (\Q_p/\Z_p))^{\vee}) \leq 1.
\]
Moreover, $(E^-(\Q_{p^g, \infty}) \otimes (\Q_p/\Z_p))^{\vee}$ is always free of rank one over $\RR_g$, while so is $(E^+(\Q_{p^g, \infty}) \otimes (\Q_p/\Z_p))^{\vee}$ if and only if $4 \nmid g$ (i.e.~$\ord_2(g) \in \{0, 1\}$).
\end{itemize}
\end{prop}

We note here that our construction of the exact sequences in claim (1) is not canonical.

The rest of this subsection is devoted to the proof of Proposition \ref{prop:83}.
We begin with reviewing the following proposition.
For the proof, we refer to Kitajima and Otsuki \cite{KO18}, which in turn is based on Kobayashi \cite{Kob03}.

For each positive integer $f$, we define $\Z_{p^f}$ as the ring of integers of $\Q_{p^f}$.
Let $\mu(\Q_{p^f})$ denote the group of roots of unity in $\Q_{p^f}$.
Let us fix a family $(\zeta_{p^{n+1}})_n$ such that $\zeta_{p^{n+1}}$ is a primitive $p^{n+1}$-th root of unity and $(\zeta_{p^{n+1}})^p = \zeta_{p^n}$ for any $n$.

\begin{prop}[{\cite[Proposition 3.4 and Corollary 3.11]{KO18}}]\label{prop:80}
Let $f \geq 1$ be an integer and we choose an element $\zeta_{(f)} \in \mu(\Q_{p^f})$ which is a basis of $\Z_{p^f}$ as a $\Z_p[\Gal(\Q_{p^f}/\Q_p)]$-module.
Then, for each $n \geq -1$, there exists a unique element $d_{f, n} \in E(\Q_{p^f, n})_p$ such that
\[
\log_{\hat{E}}(d_{f, n}) = \varepsilon_{f, n} + \sum_{j=0}^{[(n+1)/2]} (-1)^j \frac{\pi_{f, n - 2j}}{p^j},
\]
where $\log_{\hat{E}}$ denotes the logarithm map of the formal group law $\hat{E}$ associated to $E$, and we put
\[
\varepsilon_{f, n} = \sum_{j = 1}^{\infty} (-1)^{j-1} \zeta_{(f)}^{\varphi^{-(n + 1 + 2j)}} p^j
\]
and
\[
\pi_{f, n} = \zeta_{(f)}^{\varphi^{-(n+1)}} (\zeta_{p^{n+1}}-1).
\]
Moreover, the family $(d_{f, n})_{n \geq -1}$ satisfies the following.
\begin{itemize}
\item[(1)]
\[
\Tr_{\Q_{p^f, n}/\Q_{p^{f}, n-1}}(d_{f, n}) = 
\begin{cases}
	- d_{f, n-2} & (n \geq 1)\\
	- (\varphi + \varphi^{-1}) d_{f, -1} & (n = 0).
\end{cases}
\]
\item[(2)]
\[
E(\Q_{p^f, n})_p = 
\begin{cases}
	(d_{f, n}, d_{f, n-1})_{R_{f, n}} & (n \geq 0)\\
	(d_{f, -1})_{R_{f, -1}} & (n = -1).
\end{cases}
\]
\end{itemize}
\end{prop}

We should stress that the family $(d_{f, n})_{n \geq -1}$ in Proposition \ref{prop:80} depends on the choice of $\zeta_{(f)}$.
In Lemma \ref{lem:77} below, we will construct a certain compatible system $(\zeta_{(f)})_{f \geq 1}$.

\begin{lem}\label{lem:78}
There exists a family $(\zeta_{(f)})_{f \geq 1} \in \prod_{f \geq 1} \F_{p^f}^{\times}$, indexed by the positive integers $f$, satisfying the following.
\begin{itemize}
\item
For each $f \geq 1$, the element $\zeta_{(f)}$ is a basis of $\F_{p^f}$ as an $\F_p[\Gal(\F_{p^f}/\F_p)]$-module.
 \item
For each $f' \mid f$, we have $\Tr_{\F_{p^f}/\F_{p^{f'}}}(\zeta_{(f)}) = \zeta_{(f')}$.
\end{itemize}
\end{lem}

\begin{proof}
For each $f \geq 1$, it is well-known that $\F_{p^f}$ is free of rank one over $\F_p[\Gal(\F_{p^f}/\F_p)]$.
Let $B_f$ be the (non-empty) set of bases of $\F_{p^f}$ as an $\F_p[\Gal(\F_{p^f}/\F_p)]$-module.
Since $\Tr_{\F_{p^f}/\F_{p^{f'}}}: \F_{p^f} \to \F_{p^{f'}}$ is surjective for each $f' \mid f$ as is also well-known, the family $(B_f)_{f \geq 1}$ consists a projective system of sets with respect to the trace maps.
As each $B_f$ is a finite set, the projective limit $\varprojlim_{f \geq 1} B_f$ must be non-empty, and any element of the limit is what we want.
\end{proof}

\begin{lem}\label{lem:77}
There exists a family $(\zeta_{(f)})_{f \geq 1} \in \prod_{f \geq 1} \mu(\Q_{p^f})$ satisfying the following.
\begin{itemize}
\item
For each $f \geq 1$, the element $\zeta_{(f)}$ is a basis of $\Z_{p^f}$ as a $\Z_p[\Gal(\Q_{p^f}/\Q_p)]$-module.
\item
For each $f' \mid f$, there exists an element $\alpha_{f/f'} \in \Z_p[\Gal(\Q_{p^{f'}}/\Q_p)]^{\times}$ such that we have $\Tr_{\Q_{p^f}/\Q_{p^{f'}}}(\zeta_{(f)}) = \alpha_{f/f'} \zeta_{(f')}$.
\end{itemize}
\end{lem}

\begin{proof}
For each $f \geq 1$, the mod $p$ reduction map gives rise to a one-to-one correspondence
\begin{equation}\label{eq:78}
\mu(\Q_{p^f}) \simeq \F_{p^f}^{\times}.
\end{equation}
We take a family $(\zeta_{(f)})_{f \geq 1} \in \prod_{f \geq 1} \F_{p^f}^{\times}$ as in Lemma \ref{lem:78}, and then lift it to $(\zeta_{(f)})_{f \geq 1} \in \prod_{f \geq 1} \mu(\Q_{p^f})$ via the above correspondence (we abuse the notation).
Then by Nakayama's lemma, the first condition holds.
The second condition is also easy to see; we actually have $\alpha_{f/f'} \equiv 1 \mod (p)$.
\end{proof}

Now we begin the proof of Proposition \ref{prop:83}.

\begin{proof}[Proof of Proposition \ref{prop:83}]
As noted in \cite[Remark 3.4]{Kata_08}, properties (1) and (2) in Proposition \ref{prop:80} enable us to mimic the argument of \cite[\S 4.3]{Kata_08}.
As a consequence, we obtain the exact sequences claimed in Proposition \ref{prop:83}(1) for each positive integer $f$ instead of $g$.

In order to deal with a supernatural number $g$, we shall take limits with respect to positive integers $f \mid g$.
For that purpose, we make use of the system $(\zeta_{(f)})_{f \mid g} \in \prod_{f \mid g} \mu(\Q_{p^f})$ as in Lemma \ref{lem:77}, and accordingly construct a family 
\[
(d_{f, n})_{f \mid g, n \geq -1} \in \prod_{f \mid g, n \geq -1} E(\Q_{p^f, n})
\]
by Proposition \ref{prop:80}.
Then we have 
\begin{equation}\label{eq:norm_d}
\Tr_{\Q_{p^f, n}/\Q_{p^{f'}, n}}(d_{f, n}) = \alpha_{f/f'} d_{f', n}
\end{equation}
for each $f' \mid f \mid g$ and $n \geq -1$.
This is because both $(\varepsilon_{f, n})_f$ and $(\pi_{f, n})_f$ in Proposition \ref{prop:80} satisfy the corresponding relations.

Thanks to \eqref{eq:norm_d}, we have compatibility between the exact sequences for various $f$.
For instance, for the $+$ case, we have a commutative diagram
\[
\xymatrix{
	0 \ar[r] 
	& (E^+(\Q_{p^f, \infty}) \otimes (\Q_p/\Z_p))^{\vee} \ar[r] \ar@{->>}[d]_{\alpha_{f/f'}}
	& \RR_{f} \oplus R_{f, -1} \ar[r] \ar@{->>}[d]
	& R_{f, -1} \ar[r] \ar@{->>}[d]
	& 0\\
	0 \ar[r] 
	& (E^+(\Q_{p^{f'}, \infty}) \otimes (\Q_p/\Z_p))^{\vee} \ar[r]
	& \RR_{f'} \oplus R_{f', -1} \ar[r]
	& R_{f', -1} \ar[r]
	& 0
}
\]
for each $f' \mid f$.
Here, the middle and the right vertical arrows are the natural homomorphisms, while the left vertical arrow is the multiplication by $\alpha_{f/f'}$ following the natural map.
The commutativity easily follows from \eqref{eq:norm_d}, but we need to recall the detailed construction of the exact sequences, so we omit it.
Therefore, by taking the projective limit with respect to positive integers $f \mid g$, we obtain the exact sequences claimed in Proposition \ref{prop:83}(1) for a supernatural number $g$.
Note that we ignored the multiplication by $\alpha_{f/f'}$, which is possible since it is at any rate a unit and does not affect the module structure of the limit.

We briefly check that claim (2) follows from (1).
Since the $\RR_g$-module $R_{g, 0}^{\nt}$ is torsion, the statement on the generic rank is clear.
Since $\pd_{\RR_g}(R_{g, 0}^{\nt}) \leq 1$, the module $(E^-(\Q_{p^{g}, \infty}) \otimes (\Q_p/\Z_p))^{\vee}$ is free over $\RR_g$.
On the other hand, the structure of $(E^+(\Q_{p^{g}, \infty}) \otimes (\Q_p/\Z_p))^{\vee}$ depends on the endomorphism $\varphi + \varphi^{-1}$ on $R_{g, -1}$.
As in \cite[Lemma 3.6]{KO18} or \cite[Remark 4.27]{Kata_08}, it is isomorphic if and only if $4 \nmid g$.
If $4 \mid g$, the homomorphism is not injective, so $(E^+(\Q_{p^{g}, \infty}) \otimes (\Q_p/\Z_p))^{\vee}$ cannot be free.
Thus we obtain claim (2).
\end{proof}

\subsection{Consequences of Proposition \ref{prop:83}}\label{subsec:conseq_str}

We shall observe immediate consequences of Proposition \ref{prop:83}.
We continue to suppose that $E/\Q_p$ is a supersingular elliptic curve with $a_p(E) = 0$.
Let $g$ be a supernatural number.

We put
\[
D_g^{\pm} 
= \parenth{H^1(\Q_{p^g, \infty}, E[p^{\infty}]) \over E^{\pm}(\Q_{p^g, \infty}) \otimes (\Q_p/\Z_p)}^{\vee}
\]
and
\[
D_g^1 
= \parenth{H^1(\Q_{p^g, \infty}, E[p^{\infty}]) \over E(\Q_{p^g}) \otimes (\Q_p/\Z_p)}^{\vee}.
\]
These are regarded as submodules of $H^1(\Q_{p^g, \infty}, E[p^{\infty}])^{\vee}$.

\begin{prop}\label{prop:D_str}
The following are true.
\begin{itemize}
\item[(i)]
The $\RR_g$-module $H^1(\Q_{p^g, \infty}, E[p^{\infty}])^{\vee}$ is free of rank two.
\item[(ii)]
The $\RR_g$-module $D_g^{\pm}$ is free of rank one.
Moreover, $D_g^{\pm}$ is a direct summand of $H^1(\Q_{p^g, \infty}, E[p^{\infty}])^{\vee}$ if and only if either $\pm = -$ or $4 \nmid g$.
\item[(iii)]
We have
\[
D_g^1 = D_g^+ \oplus D_g^-.
\]
\end{itemize}
\end{prop}

\begin{proof}
(i)
We have $E(\Q_{p^g, \infty})[p] = 0$ because of the reduction type (e.g., \cite[Proposition 3.1]{KO18}).
It is known to experts that the claim follows from this, together with the self-duality of $T_pE$ and the local Tate duality.
We briefly explain the proof by using the notion of perfect complexes that will be introduced in the subsequent sections.
Let us consider the Iwasawa cohomology complex $\RG_{\Iw}(\Q_{p^g, \infty}, T_pE) \in D^{[0, 2]}(\RR_g) $ (see \S \ref{subsec:perf} and \S \ref{sec:CM_14}).
The local Tate duality implies $H^1(\Q_{p^g, \infty}, E[p^{\infty}])^{\vee} \simeq H^1_{\Iw}(\Q_{p^g, \infty}, T_pE)$.
Moreover, the fact $E(\Q_{p^g, \infty})[p] = 0$, together with the self-duality of $T_pE$ and the local Tate duality, implies $\RG_{\Iw}(\Q_{p^g, \infty}, T_pE) \in D^{[1, 1]}(\RR_g)$.
Then the claim follows by combining with the Euler-Poincare characteristic formula.

(ii)
This immediately follows from Proposition \ref{prop:83}(2) and claim (i).

(iii)
By Proposition \ref{prop:pm_str}, we have
\[
E(\Q_{p^g}) \otimes (\Q_p/\Z_p) 
= \parenth{E^+(\Q_{p^g, \infty}) \otimes (\Q_p/\Z_p)} 
\cap \parenth{E^-(\Q_{p^g, \infty}) \otimes (\Q_p/\Z_p)}
\]
in $E(\Q_{p^g, \infty}) \otimes (\Q_p/\Z_p)$.
This implies that we have a natural exact sequence
\[
D_g^+ \oplus D_g^- \to H^1(\Q_{p^g, \infty}, E[p^{\infty}])^{\vee} \to (E(\Q_{p^g}) \otimes (\Q_p/\Z_p))^{\vee} \to 0.
\]
By claims (i) and (ii), the first arrow is injective since $(E(\Q_{p^g}) \otimes (\Q_p/\Z_p))^{\vee}$ is torsion.
Thus we obtain the claim.
\end{proof}

\begin{rem}\label{rem:LL}
When $g = f p^{\infty}$ (resp.~$g = f$) with a positive integer $f$, Lei and Lim \cite[\S 3.2]{LL21} (resp.~Lim \cite[\S 3.2]{Lim21}) studied the structure of (the $\Gal(\Q_p(\mu_p)/\Q_p)$-invariant part of) $E^{\pm}(\Q_{p^g, \infty}) \otimes (\Q_p/\Z_p)$ in a different way.
The results in this paper are more precise as the module structures are completely determined, and it is actually possible to reprove those previous results.
\end{rem}

\subsection{The ordinary case}\label{subsec:ord_local}

As an ordinary analogue of Proposition \ref{prop:D_str}, we also have the following.
We omit the proof as it is well-known (see, e.g., \cite{Gree11}).

\begin{prop}\label{prop:D_str_ord}
Let $k$ be a finite extension of $\Q_p$ and $L/k$ an abelian $p$-adic Lie extension such that $L \supset k(\mu_{p^{\infty}})$.
Let $E/k$ be an elliptic curve with good ordinary reduction such that $E(L)[p] = 0$.
Then the following are true.
\begin{itemize}
\item[(i)]
The $\Z_p[[\Gal(L/k)]]$-module $H^1(L, E[p^{\infty}])^{\vee}$ is free of rank $2[k: \Q_p]$.
\item[(ii)]
The $\Z_p[[\Gal(L/k)]]$-module 
\[
D = \parenth{H^1(L, E[p^{\infty}]) \over E(L) \otimes (\Q_p/\Z_p)}^{\vee}
\]
is free of rank $[k: \Q_p]$ and is a direct summand of $H^1(L, E[p^{\infty}])^{\vee}$.
\end{itemize}
\end{prop}

\section{Algebraic ingredients}\label{sec:02}

In this section, we review known facts on homological algebra,
following notations in \cite{Kata_12}.

\subsection{Perfect complexes}\label{subsec:perf}

We fix notations concerning perfect complexes.

Let $R$ be a (commutative) noetherian ring.
For integers $a \leq b$, let $D^{[a, b]}(R)$ be the derived category of perfect complexes which admits a quasi-isomorphism to a complex of the form
\[
C \simeq [C^a \to C^{a + 1} \to \dots \to C^b],
\]
concentrated in degrees $a, a+1, \dots, b$, such that each $C^i$ is finitely generated and projective over $R$.
For such a complex $C$, we define the determinant of $C$ by
\[
\Det_R(C) = \bigotimes_i \Det_R^{(-1)^i}(C^i).
\]
Here, for each finitely generated projective $R$-module $F$, let $\rank_R(F)$ be the (locally constant) rank of $F$ and put
\[
\Det_R(F) = \bigwedge_R^{\rank_R(F)} F,
\qquad \Det_R^{-1}(F) = \Hom_R(\Det_R(F), R).
\]
These are invertible $R$-modules.
We also define the Euler characteristic of $C$ by
\[
\chi_R(C) = \sum_i (-1)^{i-1} \rank_R(C^i).
\]
We define $C^* = \RHom_{R}(C, R) \in D^{[-b, -a]}(R)$ by using the derived homomorphism.

\subsection{Determinant modules and Fitting ideals}\label{subsec:24}

We recall a relation between determinant modules and Fitting ideals.
See \cite[\S 3]{Kata_10} for more details.

Let $\RR$ be a ring which contains a regular local ring $\Lambda \subset \RR$ such that $\RR$ is free of finite rank over $\Lambda$.
We moreover assume that we have an isomorphism
\[
\Hom_{\Lambda}(\RR, \Lambda) \simeq \RR
\]
as $\RR$-modules.
Note that this condition implies that there is an isomorphism $\Ext^i_{\RR}(M, \RR) \simeq \Ext^i_{\Lambda}(M, \Lambda)$ for each $\RR$-module $M$.
Each ring $\RR$ defined as the Iwasawa algebra in this paper satisfies this condition.

Let $C$ be a perfect complex such that all cohomology groups of $C$ are torsion over $\RR$ (equivalently, over $\Lambda$).
Let $Q(\RR)$ be the total ring of fractions of $\RR$.
Then we have a natural homomorphism $\iota_C: \Det_{\RR}^{-1}(C) \to Q(\RR)$ defined as the composite map
\[
\iota_C: \Det_{\RR}^{-1}(C) 
\hookrightarrow Q(\RR) \otimes_{\RR} \Det_{\RR}^{-1}(C) 
\simeq \Det_{Q(\RR)}^{-1}(Q(\RR) \otimesL_{\RR} C)
\simeq Q(\RR),
\]
where $\otimesL$ denotes the derived tensor product and the last isomorphism comes from the assumption that $Q(\RR) \otimesL_{\RR} C$ is acyclic.
We put $\iDet_{\RR}(C) = \iota_C(\Det_{\RR}^{-1}(C)) \subset Q(\RR)$, which is an invertible $\RR$-submodule of $Q(\RR)$.

We have the following relation between the Fitting ideals and determinant modules.
See \cite[\S 3]{Kata_10} and \cite[just before Definition 4.5]{Kata_12}.

\begin{prop}\label{prop:75}
Let $C \in D^{[1, 2]}(\RR)$ be a complex such that $H^1(C) = 0$ and that $H^2(C)$ is torsion over $\RR$.
Then we have $\iDet_{\RR}(C) = \Fitt_{\RR}(H^2(C)) \subset \RR$.
\end{prop}

\subsection{The key algebraic proposition}\label{subsec:11}

We recall a key algebraic proposition in \cite{Kata_12}.
Let $\RR$ be as in \S \ref{subsec:24}.

\begin{prop}[{\cite[Proposition 2.1]{Kata_12}}]\label{prop:62}
Let $C \in D^{[0, 1]}(\RR)$ be a complex such that $H^0(C) = 0$.
We put $l = \chi_{\RR}(C)$.
Then we have a natural homomorphism
\[
\Psi_C: \bigwedge_{\RR}^l H^1(C) \to \Det_{\RR}^{-1}(C)
\]
such that
\[
\Ker(\Psi_C) = \left( \bigwedge_{\RR}^l H^1(C) \right)_{\tor}
\]
and
\[
\Coker(\Psi_C) \simeq \frac{\RR}{\Fitt_{\RR}(E^1(H^1(C)))}.
\]
\end{prop}

We briefly review the construction of $\Psi_C$.
We have natural maps
\[
\bigwedge_{\RR}^l H^1(C) 
\to Q(\RR) \otimes_{\RR} \bigwedge_{\RR}^l H^1(C)
\simeq \bigwedge_{Q(\RR)}^l H^1(Q(\RR) \otimesL_{\RR} C)
\simeq \Det^{-1}_{Q(\RR)}(Q(\RR) \otimesL_{\RR} C).
\]
A key point is that the image of this composite map is contained in $\Det^{-1}_{\RR}(C)$,
and then we define the map $\Psi_C$ as the induced one.
There is a generalization \cite[Proposition 2.2]{Kata_12} of this proposition, but we do not need it in this paper.

\section{Cohomological interpretations of Selmer groups and $p$-adic $L$-functions}\label{sec:CM_14}

We keep the notations in \S \ref{sec:Sel_p-L}, assuming Assumptions \ref{ass:reduction} and \ref{ass:non_anom}.
In this section, for each multi-index $\epsilon$, we introduce a complex $C_S^{\epsilon}$ that satisfies $H^2(C_S^{\epsilon}) \simeq \dSel{\epsilon}{S}$.
When $\epsilon$ is a multi-sign, we will reformulate the definition of the algebraic $p$-adic $L$-functions by using $C_S^{\epsilon}$.

We make use of well-known facts on complexes associated to Galois representations; see the book \cite{Nek06} by Nekov\'{a}\v{r} as a comprehensive reference.
The facts that we need in this paper are reviewed in \cite[\S 3.1]{Kata_12}, and we follow the notations there.

Recall that we took a set $S$ satisfying \eqref{eq:choiceS}.
Let us take an auxiliary finite set $\Sigma$ of places of $F$ such that
\[
\Sigma \supset S \cup S_p(F) \cup S_{\infty}(F)
\]
and such that $E$ has good reduction at any finite prime of $F$ not in $\Sigma$.
We define $K_{\infty, \Sigma}$ as the maximal algebraic extension of $K_{\infty}$ which is unramified outside $\Sigma$.
Note that then the module $T_pE$ is equipped with an action of $\Gal(K_{\infty, \Sigma}/F)$.
We put
\[
\Sigma_0 = \Sigma \setminus (S_p(F) \cup S_{\infty}(F)) \supset S.
\]

As in \cite[\S 3.1]{Kata_12}, we let $\RG_{\Iw}(K_{\infty, \Sigma}/K_{\infty}, T_pE)$ (resp.~$\RG_{\Iw}(K_{\infty, v}, T_pE)$ for each finite prime $v$ of $F$) be the global (resp.~local) Iwasawa cohomology complex.
Since we assume \eqref{eq:choiceS}, by \cite[Proposition 3.1]{Kata_12}, these are perfect complexes in $D^{[0, 2]}(\RR)$.
In Definition \ref{defn:C_S_e} below, we define $C_S^{\epsilon}$ by using these complexes.
Before that, we have to study the local cohomology groups for both $p$-adic and non-$p$-adic primes.

For $\pe \in S_p^{\ssr}$ and $\bullet \in \{0, 1, +, -, \rel\}$, we put
\[
D_{\pe}^{\bullet} = \left( \frac{H^1(K_{\infty, \pe}, E[p^{\infty}])}{\Loc_{\pe}^{\bullet}} \right) ^{\vee},
\]
where $\Loc_{\pe}^{\bullet}$ is defined in Definition \ref{defn:Loc}.
As an ordinary counterpart, for $\pe \in S_p^{\ord}$, we define
\[
D_{\pe} = \parenth{H^1(K_{\infty, \pe}, E[p^{\infty}]) \over E(K_{\infty, \pe}) \otimes (\Q_p/\Z_p)}^{\vee}.
\]
These are regarded as submodules of $H^1(K_{\infty, \pe}, E[p^{\infty}])^{\vee}$, which is by the local Tate duality isomorphic to $H^1_{\Iw}(K_{\infty, \pe}, T_pE)$.

We apply the results in \S \ref{sec:27} to the current semi-local setting (recall that we are assuming Assumptions \ref{ass:reduction} and \ref{ass:non_anom}).
A bit more precisely, for $\pe \in S_p^{\ssr}$, let $g$ be the residue degree of $K_{\infty}/F$ at $\pe$, which is in general a supernatural number.
By choosing a prime of $K_{\infty}$ above $\pe$, modules that we are studying are the induced modules of local counterparts associated to $\Q_{p^g, \infty}$, to which we can apply the results in \S \ref{sec:27}.
Therefore, as a consequence of Propositions \ref{prop:D_str} and \ref{prop:D_str_ord}, we obtain the following.

\begin{prop}\label{prop:CM_33}
The following are true.
\begin{itemize}
\item[(1)]
Let $\pe \in S_p^{\ssr}$.
\begin{itemize}
\item[(i)]
The $\RR$-module $H^1_{\Iw}(K_{\infty, \pe}, T_pE)$ is free of rank two.
\item[(ii)]
The $\RR$-module $D_{\pe}^{\pm}$ is free of rank one.
Moreover, $D_{\pe}^{\pm}$ is a direct summand of $H^1_{\Iw}(K_{\infty, \pe}, T_pE)$ if and only if either $\pm = -$ or the residual degree of $K_{\infty}/F$ at $\pe$ is not divisible by $4$.
\item[(iii)]
We have $D_{\pe}^1 = D_{\pe}^+ \oplus D_{\pe}^-$.
\item[(iv)]
We have $D_{\pe}^{\rel} = 0$ and $D_{\pe}^{0} = H^1_{\Iw}(K_{\infty, \pe}, T_pE)$.
\end{itemize}
\item[(2)]
Let $\pe \in S_p^{\ord}$.
\begin{itemize}
\item[(i)]
The $\RR$-module $H^1_{\Iw}(K_{\infty, \pe}, T_pE)$ is free of rank $2 \deg(\pe)$.
\item[(ii)]
The $\RR$-module $D_{\pe}$ is free of rank $\deg(\pe)$ and is a direct summand of $H^1_{\Iw}(K_{\infty, \pe}, T_pE)$.
\end{itemize}
\end{itemize}
\end{prop}

Now let $\epsilon \in \prod_{\pe \in S_p^{\ssr}} \{0, 1, +, -, \rel\}$ be a multi-index.
We put
\[
D_{S_p^{\ssr}}^{\epsilon} 
= \bigoplus_{\pe \in S_p^{\ssr}} D_{\pe}^{\epsilon_{\pe}},
\qquad
D_{S_p^{\ord}}
= \bigoplus_{\pe \in S_p^{\ord}} D_{\pe}
\]
and
\[
D_p^{\epsilon} 
= D_{S_p^{\ssr}}^{\epsilon} \oplus D_{S_p^{\ord}}.
\]
By Proposition \ref{prop:CM_33}, the $\RR$-module $D_p^{\epsilon}$ is free of rank
\begin{equation}\label{eq:Drank}
2 \times \# \{\pe \in S_p^{\ssr} \mid \epsilon_{\pe} \in \{0, 1\} \}
+ \# \{\pe \in S_p^{\ssr} \mid \epsilon_{\pe} \in \{+, -\} \}
+ \sum_{\pe \in S_p^{\ord}} \deg(\pe).
\end{equation}

Next we consider non-$p$-adic primes.

\begin{lem}\label{lem:local_cohom}
Let $v$ be a finite prime of $F$ with $v \not \in S_{\ram, p}(K_{\infty}/F) \cup S_p(F)$.
Then we have 
\[
\pd_{\RR}(H^1_{\Iw}(K_{\infty, v}, T_pE)) \leq 1,
\qquad \pd_{\RR}(H^2_{\Iw}(K_{\infty, v}, T_pE)) \leq 2.
\]
\end{lem}

\begin{proof}
Let $G_v(K_{\infty}/F)$ be the decomposition subgroup of $\Gal(K_{\infty}/F)$ at $v$, and we put $\RR_v = \Z_p[[G_v(K_{\infty}/F)]]$.
By the assumption $v \not \in S_{\ram, p}(K_{\infty}/F) \cup S_p(F)$ (and $K_{\infty} \supset \mu_{p^{\infty}}$), the topological group $G_v(K_{\infty}/F)$ is isomorphic to the product of $\Z_p$ and a finite abelian group of order prime to $p$.
Hence the algebra $\RR_v$ is a finite product of regular local rings of Krull dimension $2$.

Let us take a finite prime $w$ of $K_{\infty}$ lying above $v$.
Then, for each $i = 1, 2$, the cohomology group $H^i_{\Iw}(K_{\infty, v}, T_pE)$ is the induced module of $H^i_{\Iw}(K_{\infty, w}, T_pE)$ with respect to the ring extension $\RR_v \subset \RR$.
The observation on $\RR_v$ above immediately shows $\pd_{\RR_v}(H^i_{\Iw}(K_{\infty, w}, T_pE)) \leq 2$ for $i = 1,2$.
Moreover, it is also well-known that $\pd_{\RR_v}(H^1_{\Iw}(K_{\infty, w}, T_pE)) \leq 1$, namely, $H^1_{\Iw}(K_{\infty, w}, T_pE)$ does not contain a non-trivial finite submodule.
Therefore, we obtain the lemma.
\end{proof}

Now we begin the definition of $C_S^{\epsilon}$.
For each $\pe \in S_p(F)$, as explained in the proof of Proposition \ref{prop:D_str}(i), we have $\RG_{\Iw}(K_{\infty, \pe}, T_pE) \in D^{[1, 1]}(\RR)$, so 
\[
\RG_{\Iw}(K_{\infty, \pe}, T_pE) \simeq H^1_{\Iw}(K_{\infty, \pe}, T_pE)[-1].
\]
By Proposition \ref{prop:CM_33}, we then obtain a triangle of {\it perfect} complexes
\begin{equation}\label{eq:local_cohom_p}
D_p^{\epsilon}[-1] 
\to \RG_{\Iw}(K_{\infty} \otimes \Q_p, T_pE) 
\to {H^1_{\Iw}(K_{\infty} \otimes \Q_p, T_pE) \over D_p^{\epsilon}}[-1].
\end{equation}
By Lemma \ref{lem:local_cohom}, for $v \not \in S_{\ram, p}(K_{\infty}/F) \cup S_p(F)$, we have a triangle of {\it perfect} complexes
\begin{equation}\label{eq:local_cohom}
H^1_{\Iw}(K_{\infty, v}, T_pE)[-1] \to \RG_{\Iw}(K_{\infty, v}, T_pE) \to H^2_{\Iw}(K_{\infty, v}, T_pE)[-2].
\end{equation}

\begin{defn}\label{defn:C_S_e}
Let $\epsilon$ be a multi-index.
We define a perfect complex $C_S^{\epsilon}$ over $\RR$ by a triangle
\begin{align}
& C_S^{\epsilon}
\to \RG_{\Iw}(K_{\infty, \Sigma}/K_{\infty}, T_pE)\\
& \to \bigoplus_{v \in S} \RG_{\Iw}(K_{\infty, v}, T_pE)
\oplus {H^1_{\Iw}(K_{\infty} \otimes \Q_p, T_pE) \over D_p^{\epsilon}}[-1]
\oplus \bigoplus_{v \in \Sigma_0 \setminus S} H^2_{\Iw}(K_{\infty, v}, T_pE)[-2],
\end{align}
where the last arrow is defined by using the second arrows of triangles \eqref{eq:local_cohom_p} and \eqref{eq:local_cohom}.
\end{defn}

Note that, as the notation indicates, $C_S^{\epsilon}$ does not depend on the choice of $\Sigma$ (up to quasi-isomorphism).
This is because, if $\Sigma' \supset \Sigma$, we have a triangle 
\[
\RG_{\Iw}(K_{\infty, \Sigma}/K_{\infty}, T_pE)
\to \RG_{\Iw}(K_{\infty, \Sigma'}/K_{\infty}, T_pE)
\to \bigoplus_{v \in \Sigma' \setminus \Sigma} \RG_{\Iw, /f}(K_{\infty, v}, T_pE),
\]
and we also have a quasi-isomorphism $\RG_{\Iw, /f}(K_{\infty, v}, T_pE) \simeq H^2_{\Iw}(K_{\infty, v}, T_pE)[-2]$ for a finite prime $v \not \in S_p(F)$ of $F$ at which $E$ has good reduction.
Here, the subscript $/f$ denotes the singular part.

By the Poitou--Tate duality (e.g., \cite[Equation (3.2)]{Kata_12}) and triangles \eqref{eq:local_cohom_p} and  \eqref{eq:local_cohom}, we also have an alternative description
\begin{equation}\label{eq:PTdual_e}
C_S^{\epsilon} 
\to D_p^{\epsilon}[-1]
\oplus \bigoplus_{v \in \Sigma_0 \setminus S} H^1_{\Iw}(K_{\infty, v}, T_pE)[-1]
\to \RG(K_{\infty, \Sigma}/K_{\infty}, E[p^{\infty}])^{\vee}[-2].
\end{equation}

Now we summarize properties of $C_S^{\epsilon}$.
Note that $E(K_{\infty})[p] = 0$ holds, thanks to Assumptions \ref{ass:reduction} and \ref{ass:non_anom}.

\begin{prop}\label{prop:C_S_epsilon}
Let $\epsilon$ be a multi-index.
\begin{itemize}
\item[(1)]
We have $C_S^{\epsilon} \in D^{[1, 2]}(\RR)$, $H^2(C_S^{\epsilon}) \simeq \dSel{\epsilon}{S}$, and
\[
\chi_{\RR}(C_S^{\epsilon}) 
= \# \{\pe \in S_p^{\ssr} \mid \epsilon_{\pe} \in \{0, 1\} \} 
- \# \{\pe \in S_p^{\ssr} \mid \epsilon_{\pe} = \rel \}.
\]
\item[(2)]
If $\epsilon \in \prod_{\pe \in S_p^{\ssr}} \{+, -, \rel\}$ and Assumption \ref{ass:gen_tor} holds for at least one multi-sign $\epsilon'$ with $\epsilon' \leq \epsilon$, then we also have $H^1(C_S^{\epsilon}) = 0$.
\end{itemize}
\end{prop}

\begin{proof}
(1)
By $E(K_{\infty})[p] = 0$, we have $\RG_{\Iw}(K_{\infty, \Sigma}/K_{\infty}, T_pE) \in D^{[1, 2]}(\RR)$.
Then by Proposition \ref{prop:CM_33} and Lemma \ref{lem:local_cohom}, we obtain $C_S^{\epsilon} \in D^{[1, 2]}(\RR)$.
By \eqref{eq:PTdual_e}, we immediately obtain $H^2(C_S^{\epsilon}) \simeq \dSel{\epsilon}{S}$.
The formula about $\chi_{\RR}(C_{S}^{\epsilon})$ follows from a standard application of Euler-Poincare characteristic formulas (e.g., \cite[(7.3.1), (8.7.4)]{NSW08}), together with the formula \eqref{eq:Drank}.

(2)
For the given $\epsilon'$, we have $\chi_{\RR}(C_S^{\epsilon'}) = 0$ by (1) and so $H^1(C_S^{\epsilon'}) = 0$ as $H^2(C_S^{\epsilon'}) \simeq \dSel{\epsilon'}{S}$ is torsion.
Since we have an injective homomorphism $H^1(C_S^{\epsilon}) \hookrightarrow H^1(C_S^{\epsilon'})$, we obtain the claim.
\end{proof}

We now obtain a reformulation of the definition of the algebraic $p$-adic $L$-functions (Definition \ref{defn:CM_37}).

\begin{prop}\label{prop:CM_cohom_L}
Let $\epsilon$ be a multi-sign satisfying Assumption \ref{ass:gen_tor}.
Then we have $H^1(C_S^{\epsilon}) = 0$, $\pd_{\RR}(\dSel{\epsilon}{S}) \leq 1$, and 
\[
\iDet_{\RR}(C_{S}^{\epsilon}) = (\LL_{S}^{\epsilon}).
\]
\end{prop}

\begin{proof}
As we saw in Proposition \ref{prop:C_S_epsilon}, we have $H^1(C_S^{\epsilon}) = 0$ and so $\pd_{\RR}(H^2(C_S^{\epsilon})) \leq 1$.
The displayed formula follows from Proposition \ref{prop:75} applied to $C_{S}^{\epsilon}$.
\end{proof}

\section{The proof of the first main result}\label{sec:pf_main}

In this section, we prove Theorem \ref{thm:main} and Corollary \ref{cor:l=1}.

\subsection{The proof of Theorem \ref{thm:main}}\label{subsec:CM_15}

We mimic the proof of \cite[Theorem 5.3]{Kata_12}.
We first construct a key diagram in Proposition \ref{prop:CM_36}, and then apply the snake lemma.

As in \S \ref{subsec:main1}, let us take multi-signs $\epsilon_1, \dots, \epsilon_n$ and define $\ol{\epsilon}$ and $\ul{\epsilon}$.
For each $1 \leq i \leq n$, we define a multi-index $\delta_i = (\delta_{i, \pe})_{\pe} \in \prod_{\pe \in S_p^{\ssr}} \{+, -, \rel\}$ by 
\[
\delta_{i, \pe} = 
	\begin{cases}
	\epsilon_{i, \pe} & (\text{if $\ol{\epsilon}_{\pe} = \rel$})\\
	\rel & (\text{if $\ol{\epsilon}_{\pe} = \epsilon_{i, \pe}$}).
	\end{cases}
\]
Equivalently, $\delta_i$ is the maximum element such that $\inf \{\delta_i, \ol{\epsilon} \} = \epsilon_i$.
Using the notation in \S \ref{sec:CM_14}, we define
\[
\fD_i 
= D_{S_p^{\ssr}}^{\delta_i}.
\]
This is consistent with the definition of $\fD_i$ in \S \ref{subsec:main1}.
Then sequence \eqref{eq:DSS} follows from the definition of the Selmer groups.
Moreover, by Proposition \ref{prop:CM_33}, the module $\fD_i$ is {\it free} of rank $l$ over $\RR$.

By Definition \ref{defn:C_S_e}, we have a triangle
\begin{equation}\label{eq:DCC}
\fD_i[-2]
\to C_S^{\ol{\epsilon}} \to C_S^{\epsilon_i},
\end{equation}
the long exact sequence of which gives an exact sequence
\begin{equation}\label{eq:DSS2}
0 \to \fD_i
\to \dSel{\ol{\epsilon}}{S}
\to \dSel{\epsilon_i}{S}
\to 0.
\end{equation}
Here, the injectivity of the map $\fD_i \to \dSel{\ol{\epsilon}}{S}$ follows from Assumption \ref{ass:gen_tor} for $\epsilon_i$ and Proposition \ref{prop:C_S_epsilon}.
Then \eqref{eq:DSS2} implies that $\dSel{\ol{\epsilon}}{S}$ is generically of rank $l$ over $\RR$.

The key diagram is the following.

\begin{prop}\label{prop:CM_36}
We have a commutative diagram
\[
\xymatrix{
	\bigoplus_{i=1}^n \bigwedge_{\RR}^l \fD_i \ar[r]^-{\simeq} \ar[d]_{f_1}
	& \bigoplus_{i=1}^n \Det_{\RR} (\fD_i) \ar[d]^{f_2} \\
	\bigwedge_{\RR}^l \dSel{\ol{\epsilon}}{S} \ar[r]_-{\Psi}
	& \Det_{\RR}(C_{S}^{\ol{\epsilon}})
}
\]
which satisfies the following properties:
\[
\Ker(\Psi) = \parenth{\bigwedge_{\RR}^l \dSel{\ol{\epsilon}}{S}}_{\tor},
\qquad
\Coker(\Psi) \simeq {\RR \over \Fitt_{\RR} \parenth{E^1(\dSel{\ol{\epsilon}}{S})}},
\]
and
\[
\Coker(f_2) \simeq \frac{\RR}{\sum_{i=1}^n(\LL_{S}^{\epsilon_i})}.
\]
\end{prop}

\begin{proof}
We first construct the diagram.
We define the upper horizontal arrow as the tautological map.
By Proposition \ref{prop:C_S_epsilon}, we can apply Proposition \ref{prop:62} to $C_S^{\ol{\epsilon}}[1]$ (note that $\chi_{\RR}(C_S^{\ol{\epsilon}}[1]) = l$).
As a result we construct the map $\Psi$ in the lower horizontal arrow.
The map $f_1$ is the natural one (see \eqref{eq:DSS2}).
The map $f_2$ is defined by
\begin{equation}\label{eq:CM_40}
\Det_{\RR}(\fD_i) 
\simeq \Det_{\RR}(C_{S}^{\ol{\epsilon}}) \otimes \Det_{\RR}^{-1}(C_{S}^{\epsilon_i})
\hookrightarrow \Det_{\RR}(C_{S}^{\ol{\epsilon}}),
\end{equation}
where the isomorphism comes from the triangle \eqref{eq:DCC} and the injective map is induced by the map $\iota_{C_{S}^{\epsilon_i}}$ introduced before Proposition \ref{prop:75}.
By the constructions, it is easy to show that the diagram is commutative.

We now show the claimed properties.
The descriptions of the kernel and the cokernel of $\Psi$ are just fundamental properties in general.
Since the cokernel of \eqref{eq:CM_40} is isomorphic to 
\[
\RR/ \iDet(C_{S}^{\epsilon_i}) = \RR/( \LL_{S}^{\epsilon_i})
\]
by Proposition \ref{prop:CM_cohom_L}, we have the description of $\Coker(f_2)$.
\end{proof}

\begin{proof}[Proof of Theorem \ref{thm:main}]
By Proposition \ref{prop:CM_36}, we have a commutative diagram
\[
\xymatrix{
	&\bigoplus_{i=1}^n \bigwedge_{\RR}^l \fD_i \ar[r]^-{\simeq} \ar[d]_{f_1}
	& \bigoplus_{i=1}^n \Det_{\RR} (\fD_i) \ar[d]^{f_2} &&\\
	0 \ar[r]
	& \parenth{\bigwedge_{\RR}^l \dSel{\ol{\epsilon}}{S}}_{/\tor} \ar[r]_-{\Psi}
	& \Det_{\RR}(C_{S}^{\ol{\epsilon}}) \ar[r]
	& \Coker(\Psi) \ar[r]
	& 0
}
\]
with the lower sequence exact.
Now Theorem \ref{thm:main} follows immediately from the snake lemma applied to this diagram, by using the properties claimed in Proposition \ref{prop:CM_36}.
\end{proof}

\subsection{The proof of Corollary \ref{cor:l=1}}\label{subsec:pf_l=1}

We first show a general proposition.

\begin{prop}\label{prop:PN_tor}
Let $\RR$ be a ring as in \S \ref{subsec:24}.
Let $M$ be a finitely generated $\RR$-module with $\pd_{\RR}(M) \leq 1$.
We put $N = E^1(M)$.
Then the following hold.
\begin{itemize}
\item[(1)]
$M$ is torsion-free over $\RR$ if and only if $N$ is pseudo-null.
\item[(2)]
We suppose the equivalent conditions in (1) and moreover that the generic rank of $M$ over $\RR$ is one.
Then we have an (abstract) isomorphism
\[
\RR/\Fitt_{\RR}(N) \simeq E^2(N).
\]
\end{itemize}
\end{prop}

\begin{proof}
(1)
We regard $M$ as a module over $\Lambda$ by the forgetful functor.
Since $\pd_{\Lambda}(M) \leq 1$, we have $M_{\PN} = 0$.
This implies that $M$ is torsion-free if and only if $M_{\qu}$ is torsion-free over $\Lambda_{\qu}$ for every height one prime ideal $\qu$ of $\Lambda$.
Since $\Lambda_{\qu}$ is a discrete valuation ring, $M_{\qu}$ is torsion-free if and only if $M_{\qu}$ is free.
On the other hand, $N = E^1(M)$ is pseudo-null if and only if 
\[
\Ext^1_{\Lambda_{\qu}}(M_{\qu}, \Lambda_{\qu}) \simeq \Ext^1_{\Lambda}(M, \Lambda)_{\qu}
\]
vanishes, that is, $M_{\qu}$ is free over $\Lambda_{\qu}$ for any $\qu$ as above.
These observations imply the claim.

(2)
We first show that $E^2(N)$ is a cyclic module over $\RR$.
We take an exact sequence
\[
0 \to F \to F' \to M \to 0
\]
 with $F$ and $F'$ free modules over $\RR$ of finite ranks.
By $\rank_{\RR}(M) = 1$, we have $\rank_{\RR}(F') = \rank_{\RR}(F) + 1$.
We obtain an exact sequence
\[
0 \to M^* \to (F')^* \to F^* \to N \to 0,
\]
where in general we put $(-)^* = \Hom_{\RR}(-, \RR)$.
By the assumption that $N$ is pseudo-null, applying \cite[Proposition 2.4]{Kata_12}, we deduce that $M^*$ is a {\it free} module over $\RR$ of rank one (alternatively one may apply \cite[Proposition 3.1(c)]{Kata_13}).
Thus the above sequence is a free resolution of $N$.
Note that this implies $\pd_{\RR}(N) \leq 2$.
Moreover, the module $E^2(N)$ is a quotient of $M^{**}$, which is again free of rank one.
Therefore, $E^2(N)$ is a cyclic, as claimed.

This observation implies that we have an isomorphism $E^2(N) \simeq \RR / \Fitt_{\RR}(E^2(N))$.
On the other hand, since $N$ is pseudo-null and $\pd_{\RR}(N) \leq 2$, we have $\Fitt_{\RR}(E^2(N)) = \Fitt_{\RR}(N)$ by \cite[Proposition A.2(2)]{Kata_12}.
This completes the proof.
\end{proof}

\begin{proof}[Proof of Corollary \ref{cor:l=1}]
We have only to prove claim (2).
By Proposition \ref{prop:C_S_epsilon}, we have $\pd_{\RR}(\dSel{\ol{\epsilon}}{S}) \leq 1$.
Moreover, the generic rank of $\dSel{\ol{\epsilon}}{S}$ is $l = 1$.
Hence the corollary immediately follows from Proposition \ref{prop:PN_tor} applied to $M = \dSel{\ol{\epsilon}}{S}$.
\end{proof}

\section{The proof of the second main result}\label{sec:pf_E1}

In this section, we prove Theorem \ref{thm:E1}.

\subsection{The self-duality}\label{subsec:dual}

We keep the notations in \S \S \ref{subsec:local_settings}--\ref{subsec:conseq_str}, so
$E/\Q_p$ is a supersingular elliptic curve satisfying $a_p(E) = 0$.

The goal of this subsection is to prove the self-duality of $D_g^{\pm}$ under local duality (Proposition \ref{prop:D_inf_orth}).
The self-duality is stated in \cite[Proposition 7.11]{LP19} in the situation they deal with, but the proof has a flaw; we cannot make use of anti-symmetricity.
Instead, we adapt the argument of B.~D.~Kim \cite[Proposition 3.15]{Kim07} (a more general statement is given in \cite[Theorem 2.9]{Kim14}).
The author thanks Antonio Lei and Bharathwaj Palvannan for answering relevant questions.

We first fix notations on various pairings.
Let $f$ be a positive integer.
For each $n \geq -1$ and $j \geq 0$, we have the (perfect) local Tate pairing
\[
[-, -]_{f, n}^j: H^1(\Q_{p^f, n}, E[p^j]) \times H^1(\Q_{p^f, n}, E[p^j]) \to \Z/p^j \Z.
\]
This pairing is induced by the Weil pairing $E[p^j] \times E[p^j] \to \mu_{p^j}$.
Then we define a pairing
\[
\lra{-, -}_{f, n}^j: H^1(\Q_{p^f, n}, E[p^j])^{\iota} \times H^1(\Q_{p^f, n}, E[p^j]) \to R_{f, n}/p^j
\]
by
\[
\lra{x, y}_{f, n}^j = \sum_{\sigma \in \Gal(\Q_{p^f, n}/\Q_p)} [x^{\sigma}, y]_{f, n}^j \sigma
\]
for $x \in H^1(\Q_{p^f, n}, E[p^j])^{\iota}$ and $y \in H^1(\Q_{p^f, n}, E[p^j])$.
Note that $\lra{-, -}_{f, n}^j$ is $R_{f, n}/p^j$-bilinear because $[-, -]_{f, n}^j$ is Galois invariant.
Moreover, the pairing $\lra{-, -}_{f, n}^j$ is compatible with respect to $f$, $n$, and $j$.
Therefore, by taking the limit, we obtain a pairing
\[
\langle -, - \rangle_g: H^1_{\Iw}(\Q_{p^g, \infty}, T_pE)^{\iota} \times H^1_{\Iw}(\Q_{p^g, \infty}, T_pE) \to \RR_g
\]
for a supernatural number $g$.
This is nothing but the perfect pairing describing the local Tate duality (cf.~\eqref{eq:Tate_dual} below).

On the other hand, $[-, -]_{f, n}^j$ also induces a perfect pairing
\[
[-, -]_{g}: H^1_{\Iw}(\Q_{p^g, \infty}, T_pE) \times H^1(\Q_{p^g, \infty}, E[p^{\infty}]) \to \Q_p/\Z_p.
\]
Now we consider the submodule $D_g^{\pm}$ of $H^1(\Q_{p^g, \infty}, E[p^{\infty}])^{\vee} \simeq H^1_{\Iw}(\Q_{p^g, \infty}, T_pE)$ defined in \S \ref{subsec:conseq_str}.
By the definition, $D_g^{\pm}$ is the exact annihilator of $E^{\pm}(\Q_{p^g, \infty}) \otimes (\Q_p/\Z_p)$ with respect to $[-, -]_{g}$.

\begin{prop}\label{prop:D_inf_orth}
For a supernatural number $g$  and a choice of $\pm$, we have 
\[
\lra{(D_g^{\pm})^{\iota}, D_g^{\pm}}_g = 0.
\]
\end{prop}

\begin{proof}
First we observe that we may assume that $g = f$ is a positive integer.
This is because the pairing $\lra{-, -}_g$ and the module $D_g^{\pm}$ are both the projective limit with respect to positive integers $f \mid g$.

As in \cite[\S 3.3]{Kim07}, let us put
\[
\bH_f^{\pm} = E^{\pm}(\Q_{p^f, \infty}) \otimes (\Q_p/\Z_p) \subset H^1(\Q_{p^f, \infty}, E[p^{\infty}]),
\]
whose precise structure is described in Proposition \ref{prop:83}.
For $n \geq 0$, we put
\[
\Gamma_n = \Gal(\Q_{p^f, \infty}/\Q_{p^f, n}) \simeq \Gal(\Q_p(\mu_{p^{\infty}})/\Q_p(\mu_{p^{n+1}})).
\]

For a while, let us suppose that either $\pm = -$ or $4 \nmid f$ holds.
Then Proposition \ref{prop:83} shows that $\bH_f^{\pm}$ is a cofree module of corank one over $\RR_f$.
Therefore, we may apply \cite[Theorem 2.9]{Kim14}, which is a generalization of \cite[Proposition 3.15]{Kim07}.
As a consequence, $(\bH_f^{\pm})^{\Gamma_n}[p^j]$ is the exact annihilator of itself with respect to $[-, -]_{f, n}^j$.
Since $\bH_f^{\pm} = \varinjlim_{n, j} (\bH_f^{\pm})^{\Gamma_n}[p^j]$, by taking the limit, we obtain
\[
D_f^{\pm} = \varprojlim_{n, j} (\bH_f^{\pm})^{\Gamma_n}[p^j].
\]
On the other hand, $[(\bH_f^{\pm})^{\Gamma_n}[p^j], (\bH_f^{\pm})^{\Gamma_n}[p^j]]_{f, n}^j = 0$ also implies $\lra{(\bH_f^{\pm})^{\Gamma_n}[p^j]^{\iota}, (\bH_f^{\pm})^{\Gamma_n}[p^j]}_{f, n}^j = 0$, so taking the projective limit, we obtain $\lra{(D_f^{\pm})^{\iota}, D_f^{\pm}} = 0$ as desired.

If $\pm = +$ and $4 \mid f$, we cannot directly apply \cite[Theorem 2.9]{Kim14}.
However, by Proposition \ref{prop:83}(1), the obstruction is the failure of $\varphi + \varphi^{-1}$ to be an automorphism on $R_{f, -1}$, and can be at any rate bounded by $R_{f, -1}$ in some sense.
Moreover, the statement of the proposition is the vanishing of a submodule of $\RR_f$, so obstructions caused by torsion modules do not matter at all.
These observations enable us to modify the above argument to prove the proposition.
We omit the details, as this case is not necessary for the proof of Theorem \ref{thm:E1}.
\end{proof}

We give a corollary which will be used in the proof of Theorem \ref{thm:E1}.
Recall that $(-)^*$ denotes the linear dual.

\begin{prop}\label{prop:dual_isom}
The isomorphism
\[
H^1_{\Iw}(\Q_{p^g, \infty}, T_pE)^{\iota} 
\overset{\sim}{\to} H^1_{\Iw}(\Q_{p^g, \infty}, T_pE)^*,
\]
induced by the pairing $\lra{-, -}_g$, induces a homomorphism
\begin{equation}\label{eq:dual_hom}
{H^1_{\Iw}(\Q_{p^g, \infty}, T_pE)^{\iota} \over (D_g^{\pm})^{\iota}}
\to (D_g^{\pm})^*.
\end{equation}
Moreover, the homomorphism \eqref{eq:dual_hom} is an isomorphism if and only if either $\pm = -$ or $4 \nmid g$.
\end{prop}

\begin{proof}
The homomorphism \eqref{eq:dual_hom} is induced because of Proposition \ref{prop:D_inf_orth}.
For the final claim, we use Proposition \ref{prop:D_str}.
If \eqref{eq:dual_hom} is isomorphic, then $(D_g^{\pm})^{\iota}$ is a direct summand of $H^1_{\Iw}(\Q_{p^g, \infty}, T_pE)^{\iota}$, so either $\pm = -$ or $4 \nmid g$ holds.
If either $\pm = -$ or $4 \nmid g$ holds, then both sides of \eqref{eq:dual_hom} are free of rank one and the homomorphism is surjective, so it is isomorphic.
\end{proof}

We also record an ordinary analogue.
The proof is similar; we use Proposition \ref{prop:D_str_ord} instead of Proposition \ref{prop:D_str}.

\begin{prop}\label{prop:dual_isom_ord}
In the situation of Proposition \ref{prop:D_str_ord}, the isomorphism
\[
H^1_{\Iw}(L, T_pE)^{\iota} 
\overset{\sim}{\to} H^1_{\Iw}(L, T_pE)^*
\]
induces an isomorphism
\[
{H^1_{\Iw}(L, T_pE)^{\iota}  \over D^{\iota}} \overset{\sim}{\to} D^*.
\]
\end{prop}

\subsection{The proof of Theorem \ref{thm:E1}}\label{subsec:pf_E1}

We keep the notations in Theorem \ref{thm:E1}.
We assume Assumption \ref{ass:gen_tor}; more precisely, we only have to assume it for at least one multi-sign $\epsilon'$ with $\epsilon' \leq \epsilon$.

Recall the complex $C_S^{\epsilon}$ introduced in \S \ref{sec:CM_14}.
Then by the very definition of Ext functors, we have an isomorphism
\begin{equation}\label{eq:90}
E^1(\dSel{\epsilon}{S}) 
\simeq H^{-1} ((C_{S}^{\epsilon})^*).
\end{equation}
In Proposition \ref{prop:cpx_dual} below, we will compute $(C_{S}^{\epsilon})^*$.
We prepare preliminary local results in advance.

By the local duality (cf.~\cite[Proposition 3.3]{Kata_12}), we have
\begin{equation}\label{eq:Tate_dual}
\RG_{\Iw}(K_{\infty, v}, T_pE)^*
\simeq \RG_{\Iw}(K_{\infty, v}, T_pE)^{\iota}[2]
\end{equation}
for each finite prime $v$ of $F$.
In Lemmas \ref{lem:loc_dual_p} and \ref{lem:loc_dual} below, we observe the behavior of triangles \eqref{eq:local_cohom_p} and \eqref{eq:local_cohom} under the local duality, respectively.

\begin{lem}\label{lem:loc_dual_p}
We have a morphism between two triangles obtained by \eqref{eq:local_cohom_p}:
\[
\xymatrix{
	{H^1_{\Iw}(K_{\infty} \otimes \Q_p, T_pE) \over D_p^{\epsilon}}[-1]^* \ar[r] 
	& \RG_{\Iw}(K_{\infty} \otimes \Q_p, T_pE)^* \ar[r]
	& D_p^{\epsilon}[-1]^*\\
	D_p^{\eta, \iota}[1] \ar[r] \ar[u]
	& \RG_{\Iw}(K_{\infty} \otimes \Q_p, T_pE)^{\iota}[2] \ar[r] \ar[u]_{\vsim}
	& {H^1_{\Iw}(K_{\infty} \otimes \Q_p, T_pE)^{\iota} \over D_p^{\eta, \iota}}[1] \ar[u],
}
\]
where the middle vertical arrow is the direct sum of \eqref{eq:Tate_dual} for $v = \pe \in S_p(F)$.
Moreover, under the condition $(\star)$ in Theorem \ref{thm:E1}, the vertical arrows are all quasi-isomorphic.
\end{lem}

\begin{proof}
We apply Propositions \ref{prop:dual_isom} and \ref{prop:dual_isom_ord} to each prime $\pe \in S_p(F)$.
By taking the direct sum, we obtain the lemma.
\end{proof}

\begin{lem}\label{lem:loc_dual}
For $v \not \in S_p(F) \cup S_{\ram, p}(K_{\infty}/F)$, 
we have a morphism between two triangles obtained by \eqref{eq:local_cohom}:
\[
\xymatrix{
	H^2_{\Iw}(K_{\infty, v}, T_pE)[-2]^* \ar[r]
	& \RG_{\Iw}(K_{\infty, v}, T_pE)^* \ar[r]
	& H^1_{\Iw}(K_{\infty, v}, T_pE)[-1]^*\\
	H^1_{\Iw}(K_{\infty, v}, T_pE)^{\iota}[1] \ar[r] \ar[u]
	& \RG_{\Iw}(K_{\infty, v}, T_pE)^{\iota}[2] \ar[r] \ar[u]_{\vsim}
	&H^2_{\Iw}(K_{\infty, v}, T_pE)^{\iota}[0], \ar[u].
}
\]
where the middle vertical arrow is \eqref{eq:Tate_dual}.
Moreover, the cone of the left vertical arrow is quasi-isomorphic to $E^2(H^2_{\Iw}(K_{\infty, v}, T_pE))[0]$;
in other words, we have a triangle
\[
H^1_{\Iw}(K_{\infty, v}, T_pE)^{\iota}[1] 
\to H^2_{\Iw}(K_{\infty, v}, T_pE)[-2]^*
\to E^2(H^2_{\Iw}(K_{\infty, v}, T_pE))[0].
\]
\end{lem}

\begin{proof}
By Lemma \ref{lem:local_cohom} and the fact that $H^i_{\Iw}(K_{\infty, v}, T_pE)$ is a torsion module for $i = 1, 2$, we have
\[
H^i(H^1_{\Iw}(K_{\infty, v}, T_pE)[-1]^*) = 0
\]
for $i \neq 0$ and 
\[
H^i(H^2_{\Iw}(K_{\infty, v}, T_pE)[-2]^*) = 0
\]
 for $i \neq -1, 0$.
Thus the upper triangle of this lemma implies
\[
H^{-1}(H^2_{\Iw}(K_{\infty, v}, T_pE)[-2]^*) 
\simeq H^{-1}(\RG_{\Iw}(K_{\infty, v}, T_pE)^*)
\simeq H^1_{\Iw}(K_{\infty, v}, T_pE)^{\iota},
\]
where the final isomorphism is due to the middle vertical arrow.
This implies the existence of the left vertical arrow, and so the right vertical arrow also exists.
We have
\[
H^{0}(H^2_{\Iw}(K_{\infty, v}, T_pE)[-2]^*) 
\simeq E^2(H^2_{\Iw}(K_{\infty, v}, T_pE))
\]
by the definition of Ext functor, so the final assertion also holds.
\end{proof}

\begin{prop}\label{prop:cpx_dual}
Under the condition $(\star)$, we have a triangle
\[
C_S^{\eta, \iota} \to (C_S^{\epsilon})^*[-3]
 \to \bigoplus_{v \in S} \RG_{\Iw}(K_{\infty, v}, T_pE)^{\iota}
\oplus \bigoplus_{v \in \Sigma_0 \setminus S} E^2(H^2_{\Iw}(K_{\infty, v}, T_pE)) [-2].
\]
\end{prop}

\begin{proof}
By the linear dual of the definition of $C_S^{\epsilon}$ in Definition \ref{defn:C_S_e}, after translation, we obtain a triangle
\begin{align}\label{eq:compar}
& (C_S^{\epsilon})^*[-3]\\
&\to \bigoplus_{v \in S} \RG_{\Iw}(K_{\infty, v}, T_pE)^*[-2]
\oplus \parenth{H^1_{\Iw}(K_{\infty} \otimes \Q_p, T_pE) \over D_p^{\epsilon}}[1]^*
\oplus \bigoplus_{v \in \Sigma_0 \setminus S} H^2_{\Iw}(K_{\infty, v}, T_pE)[0]^*\\
&\to \RG_{\Iw}(K_{\infty, \Sigma}/K_{\infty}, T_pE)^*[-2].
\end{align}
We compare this with \eqref{eq:PTdual_e} for $\eta$ instead of $\epsilon$.
This is possible by \eqref{eq:Tate_dual}, Lemmas \ref{lem:loc_dual_p} and \ref{lem:loc_dual}, and a quasi-isomorphism
\[
\RG_{\Iw}(K_{\infty, \Sigma}/K_{\infty}, T_pE)^*
\simeq \RG(K_{\infty, \Sigma}/K_{\infty}, E[p^{\infty}])^{\vee}.
\]
As a result, we obtain the desired triangle.
\end{proof}

\begin{proof}[Proof of Theorem \ref{thm:E1}]
Recall the isomorphism \eqref{eq:90}.
The triangle in Proposition \ref{prop:cpx_dual} gives an exact sequence
\begin{align}
\bigoplus_{v \in S} H^1_{\Iw}(K_{\infty, v}, T_pE)^{\iota}
& \to H^2(C_S^{\eta, \iota}) \to H^{-1}((C_S^{\epsilon})^*) \\
& \to \bigoplus_{v \in S} H^2_{\Iw}(K_{\infty, v}, T_pE)^{\iota}
\oplus \bigoplus_{v \in \Sigma_0 \setminus S} E^2(H^2_{\Iw}(K_{\infty, v}, T_pE))
\to 0.
\end{align}
Since $H^2(C_S^{\eta, \iota}) \simeq (\dSel{\eta}{S})^{\iota}$, the cokernel of the first map to $H^2(C_S^{\eta, \iota})$ is isomorphic to $(\dSel{\eta}{})^{\iota}$.

As in the proof of Lemma \ref{lem:local_cohom}, for each $v \in \Sigma_0 \setminus S$, the $\RR$-module $H^2_{\Iw}(K_{\infty, v}, T_pE)$ is induced from $\RR_v$ and the Krull dimension of $\RR_v$ is two.
This implies $\pd_{\RR}(H^2_{\Iw}(K_{\infty, v}, T_pE)_{/\PN}) \leq 1$, so we have
\[
E^2(H^2_{\Iw}(K_{\infty, v}, T_pE)_{\PN})
\simeq E^2(H^2_{\Iw}(K_{\infty, v}, T_pE)).
\]
Moreover, by the local duality, we have $H^2_{\Iw}(K_{\infty, v}, T_pE) \simeq H^0(K_{\infty, v}, E[p^{\infty}])^{\vee}$ for each finite prime $v$ of $F$.
This completes the proof of Theorem \ref{thm:E1}.
\end{proof}

\section*{Acknowledgments}

I am sincerely grateful to Masato Kurihara for his constant support and encouragement during this research.
I also thank Antonio Lei and Bharathwaj Palvannan for their responses to my queries concerning their paper \cite{LP19}.
This research was supported by JSPS KAKENHI Grant Number 19J00763.

{
\bibliographystyle{abbrv}
\bibliography{biblio}
}

\end{document}